\documentclass[10pt]{amsart}
\usepackage{graphicx}
\usepackage{amscd}
\usepackage{amsmath}
\usepackage{amsthm}
\usepackage{amsfonts}
\usepackage{amssymb}
\usepackage{mathrsfs}
\usepackage{bm}
\usepackage{enumitem}
\usepackage{amsrefs}
\usepackage{xcolor}
\usepackage[colorlinks, citecolor=blue, linkcolor=red, pdfstartview=FitB]{hyperref}
\usepackage[latin1]{inputenc}

\newcommand{\bB}{{\mathbb{B}}}
\newcommand{\bC}{{\mathbb{C}}}
\newcommand{\bD}{{\mathbb{D}}}

\newcommand{\bF}{{\mathbb{F}}}

\newcommand{\bN}{{\mathbb{N}}}

  \newcommand{\A}{{\mathcal{A}}}
  \newcommand{\B}{{\mathcal{B}}}

  \newcommand{\J}{{\mathcal{J}}}

  \newcommand{\M}{{\mathcal{M}}}

  \newcommand{\Q}{{\mathcal{Q}}}

\newcommand{\fA}{{\mathfrak{A}}}

\newcommand{\fC}{{\mathfrak{C}}}
\newcommand{\fc}{{\mathfrak{c}}}

\newcommand{\fd}{{\mathfrak{d}}}

\newcommand{\fF}{{\mathfrak{F}}}

\newcommand{\fH}{{\mathfrak{H}}}

\newcommand{\fK}{{\mathfrak{K}}}
\newcommand{\fL}{{\mathfrak{L}}}
\newcommand{\fM}{{\mathfrak{M}}}
\newcommand{\fN}{{\mathfrak{N}}}

\newcommand{\fV}{{\mathfrak{V}}}
\newcommand{\fW}{{\mathfrak{W}}}

\renewcommand{\phi}{\varphi}

\newcommand{\upchi}{{\raise.35ex\hbox{$\chi$}}}

\newcommand{\ol}{\overline}

\newcommand{\bksl}{\backslash}
\newcommand{\nin}{\notin}

\newcommand{\qand}{\quad\text{and}\quad}
\newcommand{\qor}{\quad\text{or}\quad}

\newcommand{\nc}{\operatorname{nc}}
\newcommand{\comm}{\operatorname{c}}

\newcommand{\spn}{\operatorname{span}}

\newcommand{\fb}{\mathfrak{b}}
\newcommand{\fa}{\mathfrak{a}}

\newcommand{\card}{\operatorname{card }}

\newcommand{\Ann}{\operatorname{Ann}}
\newcommand{\frk}[1]{\mathfrak{#1}}

\newtheorem{lemma}{Lemma}[section]
\newtheorem{theorem}[lemma]{Theorem}

\newtheorem{corollary}[lemma]{Corollary}

\theoremstyle{definition}

\newtheorem{example}{Example}

\begin{document}
\author{Rapha\"el Clou\^atre}

\address{Department of Mathematics, University of Manitoba, Winnipeg, Manitoba, Canada R3T 2N2}

\email{raphael.clouatre@umanitoba.ca\vspace{-2ex}}
\thanks{The first author was partially supported by an NSERC Discovery Grant.}
\author{Edward J. Timko}
\email{edward.timko@umanitoba.ca\vspace{-2ex}}

\subjclass[2010]{47A13, 47A20, 47A45 (47B32, 46L52)}

\begin{abstract}
	It is known that pure row contractions with one-dimensional defect spaces can be classified up to unitary equivalence by compressions of the standard $d$-shift acting on the full Fock space. Upon settling for a softer relation than unitary equivalence, we relax the defect condition and simply require the row contraction to admit a cyclic vector. We show that cyclic pure row contractions can be ``transformed" (in a precise technical sense) into compressions of the standard $d$-shift.
Cyclic decompositions of the underlying Hilbert spaces are the natural tool to extend this fact to higher multiplicities. We show that such decompositions face multivariate obstacles of an algebraic nature.  Nevertheless, some decompositions are obtained for nilpotent commuting row contractions  by analyzing function theoretic rigidity properties of their invariant subspaces.

\end{abstract}

\title[Cyclic row contractions and rigidity of invariant subspaces]{Cyclic row contractions and \\rigidity of invariant subspaces}
\maketitle

\section{Introduction}


The structure of Hilbert space contractions has long been recognized as deeply connected to that of concrete multiplication operators on spaces of analytic functions on the unit disc.
The prototypical result in that direction is the fact that pure contractions may be viewed as compressions to co-invariant subspaces of unilateral shifts. This crucial fact was fleshed out significantly in the original edition of \cite{nagy2010}, and the resulting theory has been influential since then. 

A major theme in modern operator theory is the investigation of multivariate phenomena, prompted by the discovery of curious behaviour \cite{parrott1970},\cite{varopoulos1974}.
The multivariate counterpart to the Sz.-Nagy--Foias theory was systematically developed by Popescu in a series of papers, starting with \cite{popescu1989},\cite{popescu1991}. The setting in this case is that of row contractions, with no commutation relation assumed between the individual operators. The non-commutative multivariate theory runs parallel to the univariate one, with most classical results finding a close analogue.
The key insight into the commutative world was provided by Arveson in \cite{arveson1998} who identified the appropriate space of analytic functions on the ball to serve as a replacement for the classical Hardy space.
This space is now known as the Drury-Arveson space and is the centerpiece of many problems in function theory due its universal property \cite{AM2000}.
It should also be pointed out that the commutative world can in fact be recovered efficiently from the non-commutative one upon taking appropriate quotients, as laid out in \cite{davidson1998},\cite{popescu2006}.

Going back to the foundational single-variable result, it is known that if $T$ is a Hilbert space contraction with one-dimensional defect space (in the sense that $I-TT^*$ has rank one), then it is unitarily equivalent to a compression of the standard unilateral shift on the Hardy space of the disc. In this case, the connection with function theory is especially transparent, and tools developed in the setting of the Hardy space can be brought to bear quite successfully to understand the structure of $T$. There are corresponding results in the multivariate setting \cite{popescu1989},\cite{muller1993} that connect the study of row contractions to complex function theory in several variables. 

Unfortunately, the condition on the defect space being one-dimensional is quite restrictive, and it is not stable under natural perturbations such as scaling or similarity. It is thus desirable to replace it by a more flexible condition, such as cyclicity. The tradeoff, then, is that we cannot expect to describe $T$ up to unitary equivalence with compressions of the standard shift anymore. This relaxation idea is the basis for the very successful Jordan model theory of constrained contractions (also known as $C_0$-contractions) developed in \cite{bercovici1988}. The crucial result is that a cyclic pure contraction can be ``transformed" into a compression of the standard shift. The multivariate analogue of this relaxation procedure seems to have been overlooked, and providing it is one of the main motivation behind this project. It should be noted here that there is a well-developed model theory for arbitrary row contractions that takes place in spaces of vector-valued holomorphic functions \cite[Theorem VI.2.3]{nagy2010},\cite{bhatta2005},\cite{popescu1989char},\cite{popescu2006}. However, the focus there is different than ours, and the two strands of the theory do not overlap beyond the case of one-dimensional defect spaces.

Of course, it is desirable to deal with contractions $T$ which are not necessarily cyclic as well. The natural approach for doing so is to decompose the underlying Hilbert space into cyclic subspaces that are invariant for $T$, and then apply the cyclic result on each piece. In \cite{bercovici1988}, this is called the \emph{splitting principle}, and allows the cyclic theory to be leveraged in the general case of arbitrary multiplicity and to culminate in a powerful classification theorem that mirrors the Jordan canonical form for matrices. Inspired by this success, at the onset of this project another one of our objectives was to obtain such cyclic decompositions in the multivariate context. Surprisingly, this ambition was met with serious obstacles, as we unraveled multivariate obstructions of a seemingly purely algebraic nature. This is an interesting discovery, and shows that if a satisfactory model theory is to emerge in this context, it will have to rely on entirely new ideas and tools. At least superficially, this  is reminiscent of what prompted the work in \cite{arveson2004}.

Although cyclic decompositions appear elusive in general, we obtain some limited partial results in this direction based on considerations of independent interest on function theoretic rigidity properties of invariant subspaces. More precisely, if $\fM$ is an invariant subspace for a row contraction $T$, we are interested in describing situations where $\fM$ is completely determined by the kernel of the functional calculus associated to the restriction $T|_\fM$. We also explore the corresponding question for co-invariant subspaces. Such inquiries generate significant interest in the setting of Hilbert modules on reproducing kernel Hilbert spaces; see \cite{chenguo2003} and the references therein. Our work on this topic initiates a multivariate exploration of equivalence classes of invariant subspaces for constrained contractions, as studied in \cite{BT1991},\cite{bercovici1991},\cite{BS1996},\cite{li1999},\cite{clouatre2013invsub},\cite{clouatre2014invsub}.

We now describe the organization of the paper and state our main results. To do so, we require the following notation. For a pure commuting row contraction $T$, we denote by $\Ann_{\comm}(T)$ the set of multipliers $\phi$ on the Drury-Arveson space with the property that $\phi(T)=0$. More detail can be found in Section \ref{S:prelim}, which contains background material and gathers many preliminary results that are used throughout.

In Section \ref{S:Quasi}, we investigate the class of commuting row contractions $T$ that can be transformed into the standard commuting $d$-shift. Our analysis is based on the existence of cyclic vectors enjoying certain additional properties (Theorem \ref{T:weakorth}). One of the main results of the paper is the following (see Corollary \ref{C:cyclicann}).

\begin{theorem}\label{T:maincyclic}
Let $T=(T_1,\ldots,T_d)$ be a pure, commuting, cyclic row contraction on some Hilbert space.  Let $M$ denote the compression of the standard commuting $d$-shift associated to $\Ann_{\comm}(T)$. Then, we have that $M$ is a quasi-affine transform of $T$.
 \end{theorem}
Interestingly, despite the commuting nature of the statement, our proof hinges crucially on the non-commuting dilation results of \cite{popescu1989}.

In Section \ref{S:rigidity}, we turn to the question of rigidity of invariant and co-invariant subspaces. For general non-commuting row contractions, this problem is known to encounter significant difficulties, so we focus on the commuting setting. We illustrate that the rigidity witnessed in the univariate context can fail in several variables (Example \ref{E:invsubann}). Nevertheless, some amount of rigidity can be established in various cases. In the special case of nilpotent row contractions, our results (Theorems \ref{T:NilpNoPropSub} and \ref{T:rigidityQinv}) imply the following.

\begin{theorem}\label{T:mainrigiditynilp}
	Let $T=(T_1,\ldots,T_d)$ be a nilpotent commuting row contraction on some Hilbert space $\fH$. The following statements hold.
	\begin{enumerate}
	
	\item[\rm{(1)}] Assume that $T^*$ is cyclic. Let $\fM,\fN\subset \fH$ be invariant subspaces for $T$ such that $\Ann_{\comm}(T|_\fM)=\Ann_{\comm}(T|_\fN)$. Then, $\fM=\fN$.

	\item[\rm{(2)}] Assume that $T$ is cyclic. Let $\fM\subset \fH$ be an invariant subspace for $T$ such that $\Ann_{\comm}(T|_\fM)=\Ann_{\comm}(T)$. Then, $\fM=\fH$.
	
	\end{enumerate}
\end{theorem}

It is relevant to mention here that it is not generally true that $T^*$ is cyclic  if $T$ is a cyclic commuting nilpotent row contraction (see Example \ref{E:maxcount}). Regarding co-invariant subspaces, we mention here the finite-dimensional version of Theorem \ref{T:rigidityQcoinv}.

\begin{theorem}\label{T:mainrigidityfd}
Let $T=(T_1,\ldots,T_d)$ be a cyclic pure commuting row contraction on some finite-dimensional Hilbert space $\fH$. Let $\fM,\fN\subset \fH$ be co-invariant subspaces for $T$ such that
	$\Ann_{\comm}(P_{\fM}T|_\fM)=\Ann_{\comm}(P_{\fN}T|_\fN).$ Then, $\fM=\fN$.
\end{theorem}

Finally, in Section \ref{S:cyclicdecomp} we examine the existence of invariant decompositions. If $T=(T_1,\ldots,T_d)$ is a $d$-tuple of operators on some Hilbert space $\fH$, a pair $(\fM,\fN)$ of non-trivial invariant subspaces for $T_1,\ldots,T_d$ is an \emph{invariant decomposition} for $T$ if $\fM\cap \fN=\{0\}$ and $\ol{\fM+\fN}=\fH$. Building on the results of Section \ref{S:rigidity}, Theorem \ref{T:splitting} establishes the existence of invariant decompositions for nilpotent commuting row contractions under some additional conditions. In another direction, Example \ref{E:FromGriff} (an adaptation of an idea from \cite{griffith1969}) exhibits a non-cyclic nilpotent commuting row contraction that admits no invariant decomposition with the property that $T|_{\fM}$ is cyclic. This is surprising, as such cyclic decompositions always exist for single constrained contractions.  We shed some light on these difficulties by consider separating vectors. These are known to be abundant for single constrained contractions, yet we show that they can fail to exist in several variables, even for  nilpotent commuting row contractions on finite-dimensional spaces (Example \ref{E:maxcount}). On the other hand, we have the following (Theorem \ref{T:AmpliateMax}).

\begin{theorem}
	Let $T=(T_1,\dots,T_d)$ be a  nilpotent commuting row contraction on some Hilbert space $\fH$. Set
	\[
	\A=\bC[x_1,\dots,x_d]/(\bC[x_1,\dots,x_d]\cap \Ann_{\comm}(T)).
	\]
	Then, for any positive integer $s\geq \dim \A$, the ampliation $T^{(s)}$ has a dense set of separating vectors in $\fH^{(s)}$.

	\label{T:mainsep}
\end{theorem}

One of the recurring themes of the paper is that nilpotent commuting row contractions exhibit rich and intricate behaviours that are somehow unique to the multivariate world. Indeed, as we will see throughout, such row contractions give rise to multivariate counter-examples to classical univariate theorems of interest. This is the motivation behind our focusing on the nilpotent case in several of our main results above. While some of our proofs can doubtlessly be adapted to the wider class of commuting row contractions satisfying other types of polynomial relations, we do not to pursue such generalizations in this paper.


\section{Preliminaries}\label{S:prelim}

\subsection{The Fock spaces, the Drury-Arveson space and some associated operator algebras}\label{SS:functionspaces}

Throughout the paper, $d$ is a fixed positive integer.   Let $\bF_d^+$ denote the free semigroup on the symbols $\{1,\ldots,d\}$.
Let $\fF^2_d$ denote the \emph{full Fock space} over $\bC^d$, that is
	\[
		\fF_d^2=\bC \Omega\oplus \bigoplus_{n=1}^\infty (\bC^d)^{\otimes n}
	\]
where $\Omega$ is some distinguished unit vector. Fix an orthonormal basis $\{e_1,\ldots,e_d\}$ of $\bC^d$. If $w=k_1 k_2\ldots k_s\in \bF^+_d$, then we set
\[
e_w=e_{k_1}\otimes e_{k_2}\otimes \ldots \otimes e_{k_s}\in \fF^2_d.
\]
The set $\{e_w:w\in \bF^+_d\}$ is an orthonormal basis of $\fF^2_d$. For each $1\leq k\leq d$, we let $L_k:\fF^2_d\to \fF^2_d$ denote the left creation operator acting as
	\[
		L_k v=e_k \otimes v, \quad v\in \fF_d^2.
	\]
It is readily verified that the operators $L_1,\ldots,L_d$ are isometries with pairwise orthogonal ranges, so that
\[
\sum_{k=1}^d L_k L_k^*\leq I.
\]
The weak-$*$ closed unital subalgebra of $B(\fF^2_d)$ generated by $L_1,\ldots,L_d$ is denoted by $\fL_d$, and is usually referred to as the \emph{non-commutative analytic Toeplitz algebra} \cite{popescu1991},\cite{davidson1998alg},\cite{davidson1998}. Likewise, the norm closed unital subalgebra of $B(\fF^2_d)$ generated by $L_1,\ldots,L_d$ is denote by $\fA_d$, and is usually called the \emph{non-commutative disc algebra} \cite{popescu1991}.

The subspace $\fF^s_d\subset \fF^2_d$ spanned by the symmetric elementary tensors is the \emph{symmetric Fock space}, and it is co-invariant for $L_1,\ldots,L_d$. It was recognized by Arveson \cite{arveson1998} that $\fF^s_d$ can be identified with a space of analytic functions on the open unit ball $\bB_d\subset \bC^d$. Nowadays, this space $H^2_d$ is called the \emph{Drury-Arveson space}.

For each $1\leq k\leq d$, we denote by $x_k:\bC^d\to \bC$ the coordinate function, so that $x_k(z)=z_k$ for $z=(z_1,\ldots,z_d)\in \bC^d$. We write $\bN$ for the set of non-negative integers. We use the standard multi-index notation: for
\[
\alpha=(\alpha_1,\dots,\alpha_d)\in\bN^d
\]
we set 
\[
|\alpha|=\alpha_1+\cdots+\alpha_d,
\]
\[
\alpha!=\alpha_1!\cdots\alpha_d!
\]
and
\[ x^\alpha = x_1^{\alpha_1}\cdots x_d^{\alpha_d}. \]
The Drury-Arveson space $H^2_d$ is the Hilbert space of analytic functions $f$ on $\bB_d$ with powers series expansion
\[
f=\sum_{\alpha\in\bN^d}c_\alpha x^\alpha
\]
satisfying
\[\sum_{\alpha\in\bN^d}\frac{\alpha_1!\cdots \alpha_d!}{|\alpha|!}|c_\alpha|^2 <\infty. \]
An orthonormal basis for $H^2_d$ is given by the weighted monomials
\[
\left(\frac{|\alpha|!}{\alpha_1!\cdots \alpha_d!}\right)^{1/2}x^\alpha, \quad \alpha\in \bN^d.
\]
See \cite{arveson1998} or \cite{davidson1998} for more details. Another useful interpretation of $H^2_d$ is that it is the \emph{reproducing kernel Hilbert space} on $\bB_d$ (see \cite{agler2002}) with kernel given by the formula
\[
k(z,w)=\frac{1}{1-\langle z,w\rangle}, \quad z,w\in \bB_d.
\]
We denote by $\M_d$ the multiplier algebra of $H^2_d$. Associated to a multiplier $\phi\in \M_d$ there is a bounded linear operator $M_\phi\in B(H^2_d)$ defined as
\[
M_\phi f=\phi f, \quad f\in H^2_d.
\]
Identifying each multiplier in $\M_d$ with its associated multiplication operator, we may view $\M_d$ as a weak-$*$ closed unital subalgebra of $B(H^2_d)$. In particular, $\M_d$ is an operator algebra. Under this identification, the multiplier norm of $\phi\in \M_d$ is simply the operator norm of $M_\phi$. 
Every polynomial is a multiplier, and we denote by $\A_d$ the norm closure of the polynomials in $\M_d$. We will require the following basic fact about multipliers.

\begin{theorem}\label{T:gleasontrick}
	Let $\phi\in \M_d$. Fix $n\in \bN$ and $w\in \bB_d$. For each $\alpha\in \bN^d$ with $|\alpha|=n$, there exists a multiplier $\phi_\alpha\in \M_d$ with the property that
	\[ \phi=\sum_{|\alpha|<n}  \frac{\partial^\alpha \phi}{\partial x^\alpha}(w)\frac{(x-w)^\alpha}{\alpha!}+\sum_{|\alpha|=n} (x-w)^\alpha \phi_\alpha. \]
	\end{theorem}
\begin{proof}
Apply  \cite[Cor. 4.2]{gleason2005} inductively.
\end{proof}

Next, we record useful identifications which are well known. The details of the proof appear to be difficult to track down explicitly, so we provide them for the reader's convenience.

\begin{theorem}\label{T:quotientcomm}
The following statements hold.
\begin{enumerate}

\item[\rm{(1)}]  There is a unitary operator $U:H^2_d\to \fF^s_d$ with the property that
\[
UM_{x_k}U^*=P_{\fF^s_d}L_k|_{\fF^s_d}, \quad 1\leq k\leq d.
\]

\item[\rm{(2)}] Let $\fW_d\subset \fL_d$ denote the weak-$*$ closure of the commutator ideal of $\fL_d$.Then, there is a unital completely isometric and weak-$*$ homeomorphic algebra isomorphism
\[
\Psi:\fL_d/\fW_d\to \M_d
\]
such that 
\[
\Psi(L_k+\fW_d)=U^*(P_{\fF^s_d}L_k|_{\fF^s_d})U=M_{x_k}
\]
for every $1\leq k\leq d$.

\item[\rm{(3)}] Let $\fC_d\subset \fA_d$ denote the norm closure of the commutator ideal of $\fA_d$. Then, there is a unital completely isometric algebra isomorphism
\[
\Phi:\fA_d/\fC_d\to \A_d
\]
such that 
\[
\Phi(L_k+\fC_d)=M_{x_k}, \quad 1\leq k\leq d.
\]

\end{enumerate}
\end{theorem}
\begin{proof}
(1) The existence of $U$ follows from the discussion in  \cite[Section 1]{arveson1998}.

(2) This follows from (1) and  \cite[Corollary 2.3]{davidson1998}.

(3)  This argument is based on the idea behind the proof of \cite[Corollary 8.3]{CH2018}. Let $q:\fA_d\to \fA_d/\fC_d$ denote the quotient map. By the Blecher--Ruan--Sinclair characterization of operator algebras \cite{BRS1990}, there is a Hilbert space $\fH$ and a unital completely isometric homomorphism $\pi: \fA_d/\fC_d\to B(\fH)$. Then, 
\[
(\pi(q(L_1)),\ldots, \pi(q(L_d)))
\]
is a \emph{commuting} row contraction on $\fH$. In particular, invoking \cite[Theorem 6.2]{arveson1998} for instance, we obtain a unital completely contractive homomorphism 
\[
\sigma:\A_d\to \pi(\fA_d/\fC_d)
\]
such that $\sigma(M_{x_k})=\pi(q(L_k))$ for every $1\leq k\leq d$. 
Next, let $\iota:\fA_d/\fC_d\to \fL_d/\fW_d$ be defined as
\[
\iota(T+\fC_d)=T+\fW_d, \quad T\in \fA_d.
\]
It is readily checked that $\fC_d$ is the norm closure of the ideal generated by
\[
\{L_jL_k-L_kL_j:1\leq j,k\leq d\}
\]
inside of $\fA_d$. Furthemore, it follows from \cite[Proposition 2.4]{davidson1998alg} that $\fW_d$ is the weak-$*$ closure of  the ideal generated by
\[
\{L_jL_k-L_kL_j:1\leq j,k\leq d\}
\]
inside of $\fL_d$. Hence, $\fC_d\subset \fW_d$ so that $\iota$ is a well-defined  unital completely contractive homomorphism. We note that
\[
(\Psi\circ \iota\circ  \pi^{-1}\circ \sigma)(M_{x_k})=M_{x_k}
\]
for every $1\leq k\leq d$, whence
\[
\Psi\circ \iota\circ \pi^{-1}\circ \sigma
\]
is simply the inclusion of $\A_d$ inside $\M_d$. This forces $\pi^{-1}\circ\sigma:\A_d\to \fA_d/\fC_d$ to be completely isometric and in particular surjective, since its image contains $L_k+\fC_d$ for every $1\leq k\leq d$. Take $\Phi=\pi^{-1}\circ \sigma$.
\end{proof}

When $d=1$, the space $H^2_1$ coincides with the usual Hardy space on the open unit disc $\bD$. In this case, the operator algebras $\A_1$ and $\M_1$ can be identified with the disc algebra $A(\bD)$ and the algebra of bounded analytic functions $H^\infty(\bD)$, respectively. It is important to note that when $d>1$, however, the quantities
\[
\|M_\phi\| \qand \|\phi\|_\infty=\sup_{z\in \bB_d}|\phi(z)|
\]
are in general not comparable \cite{arveson1998},\cite{davidson1998}. This makes the function theory in $H^2_d$ and $\M_d$ less transparent than it is in the algebra of bounded holomorphic functions $H^\infty(\bB_d)$. 

\subsection{Dilations to $d$-shifts}\label{SS:dilation}

We recall the main features of the multivariate generalization of the classical Sz.-Nagy--Foias model theory from \cite{nagy2010}, starting with the non-commuting setting.  define the \emph{standard $d$-shift} to be the $d$-tuple of operators $L=(L_1,L_2, \ldots, L_d)$ acting on $\fF^2_d$. More generally, a $d$-tuple $V=(V_1,\ldots,V_d)$ of operators on some Hilbert space $\fH$ is called a \emph{$d$-shift} if there is another Hilbert space $\fK$ and a unitary operator $U:\fH\to \fF^2_d\otimes \fK$ such that
\[
U V_k U^*=L_k\otimes I_\fK, \quad 1\leq k\leq d.
\]
The cardinal number $\dim \fK$ is then called the \emph{multiplicity} of the $d$-shift. Interestingly, the multiplicity of a $d$-shift can be interpreted in a different way as the next theorem shows. Recall that a subset $\Gamma\subset \fH$ is \emph{cyclic} for a $d$-tuple $T=(T_1,\ldots,T_d)$ if
\[
\fH=\ol{\spn \{T_w \gamma:w\in \bF^+_d, \gamma\in \Gamma\}}.
\]
Here, given a word $w=i_1i_2\ldots i_s\in \bF^+_d$, we use the notation
\[
T_w=T_{i_1}T_{i_2}\ldots T_{i_s}.
\]
\begin{theorem}\label{T:shiftmult}
Let $V=(V_1,\ldots.V_d)$ be a $d$-shift. Then, the multiplicity of $V$ coincides with the least cardinality of a cyclic set for $V$.
\end{theorem}
\begin{proof}
This follows from \cite[Corollary 1.10]{popescu2006}.
\end{proof}
The importance of $d$-shifts lies in the fact that they can be used to model general $d$-tuples of operators satisfying some natural necessary conditions, as we now illustrate.

A $d$-tuple $T=(T_1,\ldots,T_d)$ of operators on some Hilbert space $\fH$ is said to be a \emph{row contraction} if the row operator 
\[
[T_1,\ldots, T_d]:\fH^{(d)}\to \fH
\]
defined as
\[
[T_1,\ldots, T_d]\xi=\sum_{k=1}^d T_k \xi_k, \quad \xi=(\xi_1,\ldots,\xi_d)\in \fH^{(d)}
\]
is contractive. Here, we use the standard notation
\[
\fH^{(d)}=\fH\oplus \fH\oplus \ldots \oplus \fH.
\]
Equivalently, $T$ is a row contraction if and only if
\[
\sum_{k=1}^d T_k T_k^*\leq I. 
\]
The row contraction $T$ is said to be \emph{pure} if
\[ \lim_{n\to\infty} \sum_{|w|=n} \|T_w^*\xi\|^2 = 0, \quad \xi\in \fH. \]
Furthermore, we define the \emph{defect space} of $T$ to be 
\[
\Delta_T=\ol{\left(I-\sum_{k=1}^d T_k T_k^* \right)\fH},
\]
and the \emph{defect} of $T$ to be 
\[
\fd_T=\dim \Delta_T.
\]
A straightforward calculation reveals that $d$-shifts are pure row contractions. The following result shows that in some sense they are the universal example.

\begin{theorem}\label{T:popescudilation}
Let $T=(T_1,\ldots,T_d)$ be a $d$-tuple of operators on some Hilbert space $\fH$. Then, the following statements are equivalent.
\begin{enumerate}

\item[\rm{(i)}] The $d$-tuple $T$ is a pure row contraction.

\item[\rm{(ii)}] There is another Hilbert space $\fK$ containing $\fH$ along with a $d$-shift $V=(V_1,\ldots,V_d)$ on $\fK$ such that
	\[
		T^*_k=V_k^*|_{\fH}, \quad 1\leq k\leq d. 
	\]
\end{enumerate}
In this case, the multiplicity of the $d$-shift $V$ is $\fd_T$.
\end{theorem}
\begin{proof}
This is  \cite[Theorems 2.1 and 2.8]{popescu1989}.
\end{proof}

Of particular interest for us in this paper are row contractions $T=(T_1,\dots,T_d)$ with the property that $T_1,\dots,T_d$ commute with each other.
In this case, we say that $T$ is a \emph{commuting} row contraction. While Theorem \ref{T:popescudilation} applies just as well to commuting row contractions, it fails to encode the commutativity. In particular, one may want to replace the $d$-shift therein by some ``commutative" version which acts on the \emph{symmetric} Fock space (or Drury-Arveson space), as opposed to the full Fock space. To make this precise, recall that for each $1\leq k\leq d$ the operator $M_{x_k}\in B(H^2_d)$ denotes the multiplication operator by the coordinate function $x_k$. A standard calculation reveals that the commuting $d$-tuple $M_x=(M_{x_1},\ldots,M_{x_d})$ is a pure row contraction. We will refer to this commuting row contraction as the \emph{standard commuting $d$-shift}. More generally, a commuting $d$-tuple $Z=(Z_1,\ldots,Z_d)$ of operators on some Hilbert space $\fH$ is called a \emph{commuting $d$-shift} if there is another Hilbert space $\fK$ and a unitary operator $U:\fH\to H^2_d\otimes \fK$ such that
\[
U Z_k U^*=M_{x_k}\otimes I_\fK, \quad 1\leq k\leq d.
\]
The cardinal number $\dim \fK$ is then called the \emph{multiplicity} of the commuting $d$-shift. Just like their non-commuting counterparts, commuting $d$-shifts occupy a universal role among pure commuting row contractions.

\begin{theorem}\label{T:MVdilation}
Let $T=(T_1,\ldots,T_d)$ be a $d$-tuple of operators on some Hilbert space $\fH$. Then, the following statements are equivalent.
\begin{enumerate}

\item[\rm{(i)}] The $d$-tuple $T$ is a pure commuting row contraction.

\item[\rm{(ii)}] There is another Hilbert space $\fK$ containing $\fH$ along with a commuting $d$-shift $Z=(Z_1,\ldots,Z_d)$ on $\fK$ such that
	\[
		T^*_k=Z_k^*|_{\fH}, \quad 1\leq k\leq d. 
	\]
\end{enumerate}
In this case, the multiplicity of the commuting $d$-shift $Z$ is $\fd_T$.
\end{theorem}
\begin{proof}
This is  \cite[Theorem 4.5]{arveson1998} (alternatively, see \cite[Theorem 2.1]{popescu2006}).
\end{proof}

\subsection{Functional calculus and constrained row contractions}\label{SS:functcalc}

In \cite{popescu1991}, Popescu extends the classical von Neumann inequality to the setting of row contractions and defines an appropriate functional calculus, as follows.

\begin{theorem}\label{T:functcalcnc}
Let $T=(T_1,\ldots,T_d)$ be a row contraction on some Hilbert space $\fH$. Then, there is a unital completely contractive homomorphism $\alpha_T:\fA_d\to B(\fH)$ such that $\alpha_T(L_k)=T_k$ for every $1\leq k\leq d$. If $T$ is pure, then $\alpha_T$ extends to a unital completely contractive and weak-$*$ continuous homomorphism $\beta_T:\fL_d\to B(\fH)$.
\end{theorem}
\begin{proof}
Combine \cite[Theorems 3.6 and 3.9]{popescu1991} with \cite[Theorem 2.1]{popescu1996}.
\end{proof}

 There is a version of this result adapted to the commuting setting, and the resulting homomorphisms are naturally compatible.
 
\begin{corollary}\label{C:functcalccomm}
Let $T=(T_1,\ldots,T_d)$ be a commuting row contraction on some Hilbert space $\fH$. Let $\fC_d\subset \fA_d$ and $\fW_d\subset \fL_d$ denote respectively the norm closure and the weak-$*$ closure of the commutator ideals. Let $q:\fA_d\to \fA_d/\fC_d$ and $q_w:\fL_d\to \fL_d/ \fW_d$ be the corresponding quotient maps. Let $\Phi:\fA_d/\fC_d\to \A_d$ and $\Psi:\fL_d/\fW_d\to \M_d$ be the canonical isomorphisms from Theorem \ref{T:quotientcomm}. Then, the following statements hold.
\begin{enumerate}

\item[\rm{(1)}] There is a unital completely contractive homomorphism $\widehat{\alpha_T}:\A_d\to B(\fH)$ such that $\widehat {\alpha_T}(M_{x_k})=T_k$ for every $1\leq k\leq d$. Moreover, we have $\alpha_T=\widehat{\alpha_T}\circ \Phi\circ q$. 

\item[\rm{(2)}] If $T$ is pure, then $\widehat{\alpha_T}$ extends to a unital completely contractive and weak-$*$ continuous homomorphism $\widehat{\beta_T}:\M_d\to B(\fH)$. Moreover, we have  $\beta_T=\widehat{\beta_T}\circ \Psi\circ q_w$. 
\end{enumerate}
\end{corollary}
\begin{proof}
Since $T$ is commuting, we see that $\fC_d\subset \ker \alpha_T$, so there is a unital completely contractive homomorphism $\alpha_T':\fA_d/\fC_d\to B(\fH)$ such that $\alpha_T=\alpha'_T\circ q$. Put $\widehat{\alpha_T}=\alpha'_T\circ \Phi^{-1}$, and (1) follows. 

For (ii), we note that $\fW_d\subset \ker \beta_T$, so there is a unital completely contractive homomorphism $\beta_T':\fL_d/\fW_d\to B(\fH)$ such that $\beta_T=\beta'_T\circ q_w$. 
A standard compactness argument using the Krein-Smulyan theorem shows that $\beta'_T$ is in fact weak-$*$ continuous.
Statement (2) follows upon setting $\widehat{\beta_T}=\beta'_T\circ \Psi^{-1}$.
\end{proof}

We typically make no explicit mention of the homomorphisms $\alpha_T, \widehat{\alpha_T},\beta_T$ and $\widehat{\beta_T}$ defining the functional calculi, and for an appropriate element $\theta$ we simply write $\theta(T)$.

The existence of a weak-$*$ continuous extension to $\M_d$ of the $\A_d$--functional calculus for a commuting row contraction does not require $T$ to be pure. Indeed, the class of so-called \emph{absolutely continuous} commuting row contractions is  larger \cite{CD2016abscont},\cite{CD2016duality}. Nevertheless, for the purposes of this work we will restrict our attention to the pure case.

Let $T=(T_1,\ldots,T_d)$ be a pure row contraction. Then, the \emph{non-commutative annihilator} of $T$ is the weak-$*$ closed two-sided ideal of $\fL_d$ given by
\[ \Ann_{\nc}(T)=\{\theta\in\fL_d: \theta(T)=0\}. \]
We say that $T$ is \emph{constrained} if $\Ann_{\nc}(T)$ is non-trivial.
If $T$ happens to be commuting, then its \emph{commutative annihilator} is  the weak-$*$ closed  ideal of $\M_d$ given by
\[
\Ann_{\comm}(T)=\{\theta\in \M_d:\theta(T)=0\}.
\]
It follows from Corollary \ref{C:functcalccomm} that 
\[
\Ann_{\nc}(T)=(\Psi\circ q_w)^{-1}(\Ann_{\comm}(T))
\]
so in particular
\begin{equation}\label{Eq:ann}
\Ann_{\comm}(T)=(\Psi \circ q_w)(\Ann_{\nc}(T)).
\end{equation}
In the classical case where $d=1$, a constrained contraction $T$ is usually said to be \emph{of class $C_0$} \cite{bercovici1988}. We feel our choice of terminology is a bit more descriptive, and it is consistent with its use in \cite{popescu2006}. We also mention that an absolutely continuous contraction for which the associated $H^\infty$--functional calculus has non-trivial kernel is automatically pure, and thus constrained. To see this, combine the proof of \cite[Lemma II.1.12]{bercovici1988} with \cite[Theorem 3.2]{CD2016abscont}. For commuting row contractions, the situation is more complicated: see  \cite[Theorem 5.4]{CD2016abscont} and the discussion that follows it.
Let $\fb\subset \fL_d$ be a weak-$*$ closed two-sided ideal. We let
\[
[\fb \fF^2_d]=\ol{\spn\{\theta v:\theta\in \fb, v\in \fF^2_d\}}
\]
and
\[
\fF_\fb=\fF^2_d\ominus [\fb \fF^2_d].
\]
Then, $\fF_\fb$ is co-invariant for $\fL_d+\fL_d'$, where $\fL_d'\subset B(\fF^2_d)$ denotes the commutant of $\fL_d$. Define
\[
L_{\fb}=(P_{\fF_\fb}L_1|_{\fF_\fb},\ldots,P_{\fF_\fb}L_1|_{\fF_\fb}).
\]
It is readily verified that $L_{\fb}$ is a pure row contraction, and it is a consequence of the following result that $\Ann_{\nc}(L_\fb)=\fb$. Thus, the row contraction $L_{\fb}$ provides a basic example of a constrained row contraction.

 \begin{lemma}\label{L:anncoinv}
Let $\fK\subset \fF^2_d$ be a closed subspace which is co-invariant for $\fL_d+\fL_d'$. Let $\fb=\Ann_{\nc}(P_{\fK}L_1|_{\fK},\ldots,P_{\fK}L_d|_{\fK})$. Then, $\fb$ is the unique weak-$*$ closed two-sided ideal of $\fL_d$ with the property that $\fK=\fF_\fb$.
\end{lemma}
\begin{proof}
It follows from \cite[Theorem 2.1]{davidson1998alg} that there is a weak-$*$ closed two-sided ideal $\fc\subset \fL_d$ such that $\fF^2_d\ominus \fK=[\fc \fF^2_d]$, whence $\fK=\fF_\fc$. It only remains to show uniqueness. Assume that  $\fa\subset \fL_d$ is a weak-$*$ closed two-sided ideal such that $\fF^2_d\ominus \fK=[\fa \fF^2_d]$. Since $\fb$ was chosen to be the annihilator of $(P_{\fK}L_1|_{\fK},\ldots,P_{\fK}L_1|_{\fK})$, we see that $P_{\fK}\theta|_{\fK}=0$ for every $\theta\in \fb$. Using that $\fK$ is co-invariant for $\fL_d$ we obtain  $\theta\fF^2_d\subset \fF^2_d\ominus \fK$ for every $\theta\in \fb$. We conclude that 
\[
[\fb \fF^2_d]\subset \fF^2_d\ominus \fK=[\fa \fF^2_d].
\]
An application of \cite[Theorem 2.1]{davidson1998alg} yields $\fb\subset \fa$. On the other hand, it is immediate that $\fa$ annihilates $(P_{\fK}L_1|_{\fK},\ldots,P_{\fK}L_d|_{\fK})$, whence $\fa\subset \fb$. 
\end{proof}

We mention here that the condition that $\fK$ be co-invariant for $\fL_d'$ in the previous lemma cannot simply removed. Indeed, by  \cite[Theorem 3.7]{kribs2001Ld} there is a proper closed subspace $\fK\subset \fF^2_d$ which is co-invariant for $\fL_d$ and which satisfies
\[
\{\theta\in \fL_d:\theta \Omega\subset \fK^\perp\}=\{0\}
\]
so in particular we see that 
\[
\{\theta\in \fL_d:\theta \fF^2_d\subset \fK^\perp\}=\{0\}.
\]
Thus
\[
\Ann_{\nc}(P_{\fK}L_1|_{\fK},\ldots,P_{\fK}L_1|_{\fK})=\{0\},
\]
yet $\fK\neq \fF^2_d$.

We now examine the commuting case. Let $\J\subset \M_d$ be a weak-$*$ closed two-sided ideal and let
\[
H_\J=H^2_d\ominus [\J H^2_d]
\]
In this setting, the analogue of Lemma \ref{L:anncoinv} holds without additional conditions on the co-invariant subspace.

\begin{lemma}\label{L:anncoinvcomm}
Let $\fK\subset H^2_d$ be a closed subspace which is co-invariant for $\M_d$. Let $\J=\Ann_{\comm}(P_{\fK}M_{x_1}|_{\fK},\ldots,P_{\fK}M_{x_d}|_{\fK}).$ Then, $\J$ is the unique weak-$*$ closed ideal  of $\M_d$ with the property that $\fK=H_\J$.
\end{lemma}
\begin{proof}
Repeat the argument used in the proof of Lemma \ref{L:anncoinv}, invoking  \cite[Theorem 2.4]{DRS2015} rather than  \cite[Theorem 2.1]{davidson1998alg}.
\end{proof}

Define a pure commuting row contraction as
\[
M_\J=(P_{H_\J}M_{x_1}|_{H_\J},\ldots,P_{H_\J}M_{x_d}|_{H_\J}).
\]
We record an important property of $M_\J$. 

\begin{lemma}\label{L:uecomm}
Let $\J\subset \M_d$ be a weak-$*$ closed ideal. Let 
\[
\fb=(\Psi\circ q_w)^{-1}(\J)\subset \fL_d.
\]
Then, $M_\J$ is unitarily equivalent to $L_\fb$ and $\Ann_{\comm}(M_\J)=\J.$
\end{lemma}
\begin{proof}
By virtue of Lemma \ref{L:anncoinvcomm} we see that $\Ann_{\comm}(M_\J)=\J.$ Next, it follows from Theorem \ref{T:quotientcomm} that
\[
U\J U^*=\{ P_{\fF^s_d}X|_{\fF^s_d}: (\Psi\circ q_w)(X)\in \J\}=\{ P_{\fF^s_d}X|_{\fF^s_d}: X\in \fb\}
\]
whence 
\[
U[\J H^2_d]=\ol{\spn\{P_{\fF^s_d}Xv:X\in \fb, v\in \fF^s_d\}}=P_{\fF^s_d}[\fb \fF^2_d].
\]
Clearly, we see that $\fW_d\subset \fb$. Invoking \cite[Proposition 2.4]{davidson1998alg}, we find
\[
(\fF^s_d)^\perp=[\fW_d \fF^2_d]\subset [\fb \fF^2_d]
\] 
and thus
\[
 [\fb \fF^2_d]=U[\J H^2_d] \oplus (\fF^s_d)^\perp.
\]
In particular, we find
\[
\fF_\fb= [\fb \fF^2_d]^\perp=\fF^s_d\ominus U[\J H^2_d]=U(H^2_d\ominus [\J H^2_d])=UH_\J.
\]
We conclude that $M_\J$ is unitarily equivalent to $L_{\fb}$ via the operator $U$. 
\end{proof}

\subsection{Quasi-affine transforms}\label{SS:transforms}

Let $T=(T_1,\ldots,T_d)$ and $T'=(T'_1,\ldots,T'_d)$  be two $d$-tuples of operators acting on Hilbert spaces $\fH$ and $\fH'$ respectively. We say $T'$ is a \emph{quasi-affine transform} of $T$ and write $T'\prec T$, if there exists an injective bounded linear operator $X:\fH'\to \fH$ with dense range such that $XT_k'=T_kX$ for $1\leq k\leq d$. The relationship $T'\prec T$ is a common weakening of similarity that we will play a prominent role in our work. 

\begin{lemma}\label{L:denserangeinj}
Let $\fH$ and $\fH'$ be Hilbert spaces and let $X:\fH'\to\fH$ be a bounded linear operator.  Let $T\in B(\fH)$ and $T'\in B(\fH')$ such that $XT'=TX$. Then, the subspace $\fH'\ominus \ker X$ is co-invariant for $T'$ and
\[
XP_{\fH'\ominus \ker X}T'|_{\fH'\ominus \ker X}=TX|_{\fH'\ominus \ker X}.
\]
In particular, if $X$ has dense range, then 
\[
P_{\fH'\ominus \ker X}T'|_{\fH'\ominus \ker X}\prec T.
\]
\end{lemma}
\begin{proof}
Note that 
\[
XT' \ker X=TX\ker X=\{0\}
\]
whence $T'\ker X\subset \ker X$. We conclude that $\fH'\ominus \ker X$ is co-invariant for $T'$. Furthermore, we have
\begin{align*}
XP_{\fH'\ominus \ker X}T'|_{\fH'\ominus \ker X}&=XT'|_{\fH'\ominus \ker X}=TX|_{\fH'\ominus \ker X}.
\end{align*}
Finally, since $X|_{\fH'\ominus \ker X}:\fH'\ominus \ker X\to \fH$ is injective and has the same range as $X$, we conclude that 
\[
P_{\fH'\ominus \ker X}T'|_{\fH'\ominus \ker X}\prec T
\]
whenever $X$ has dense range.
\end{proof}

Annihilators are well behaved under quasi-affine transforms, and we record the following easy fact.

\begin{lemma}\label{L:anntransf}
Let $T=(T_1,\ldots,T_d)$ and $T'=(T'_1,\ldots,T'_d)$ be pure row contractions on some Hilbert spaces $\fH$ and $\fH'$.  Let $X:\fH'\to\fH$ be a bounded linear operator such that $XT'_k=T_kX$ for every $1\leq k\leq d$. Then, following statements hold.
\begin{enumerate}

\item[\rm{(1)}] If $X$ is injective, then $\Ann_{\nc}(T)\subset \Ann_{\nc}(T')$.

\item[\rm{(2)}] If $X$ has dense range, then $\Ann_{\nc}(T')\subset \Ann_{\nc}(T)$.

\item[\rm{(3)}] If $X$ is injective and has dense range, then $\Ann_{\nc}(T)=\Ann_{\nc}(T')$.

\end{enumerate}
\end{lemma}
\begin{proof}
This easily follows from the fact that for every $\theta\in \fL_d$ we have $X\theta(T')=\theta(T)X$. 
\end{proof}


\section{Transforming the standard $d$-shift}\label{S:Quasi}

In order to understand general pure row contractions, it is natural to turn to Theorem \ref{T:popescudilation}. This result reduces the problem  to the study of compressions of $d$-shifts to co-invariant subspaces. Such a strategy is appealing, as the non-commutative function theory of $\fL_d$ that was significantly developed in \cite{davidson1998alg},\cite{davidson1998},\cite{davidson1999} could then be leveraged, much in the way that the function theory of $H^\infty(\bD)$ has helped single operator theory blossom \cite{nagy2010}. Unfortunately, in general the $d$-shift from Theorem \ref{T:popescudilation}  is of high multiplicity, and thus the link with $\fL_d$ is obscured. In this section, to circumvent this difficulty we investigate the possibility for $d$-shifts of multiplicity one to be quasi-affine transforms of pure row contractions, taking inspiration from and generalizing a known result for single operators \cite[Theorem I.3.7]{bercovici1988}.

We first exhibit a necessary and sufficient condition in the form of the existence of certain cyclic vectors, the images of which satisfy some weak form of orthogonality. Recall that given vectors $\xi,\eta\in \fH$, we denote by $\xi\otimes \eta\in B(\fH)$ the rank-one operator defined as
	\[
		(\xi\otimes \eta)h=\langle h,\eta \rangle \xi, \quad h\in \fH.
	\]

\begin{theorem}\label{T:weakorth}
	Let $T=(T_1,\ldots,T_d)$ be a row contraction on some Hilbert space $\fH$ and let $C>0$.
	The following statements are equivalent.
\begin{enumerate}

	\item[\rm{(i)}] There is a cyclic vector $\xi\in \fH$ for $T$ that satisfies
	\[
		\sum_{w\in\bF_d^+} T_w \xi \otimes T_w \xi\leq C^2 I.
	\] 

	\item[\rm{(ii)}] There exists an operator $X:\fF_d^2\to \fH$ with dense range such that  $\|X\|\leq C$ and $XL_k=T_kX$ for $1\leq k\leq d$.

	\item[\rm{(iii)}] There exists a closed subspace $\fV\subset \fF^2_d$ which is co-invariant for $L$ and an injective operator $X:\fV\to \fH$ with dense range such that  $\|X\|\leq C$ and 
	\[
	XP_{\fV}L_k|_{\fV}=T_kX
	\]
	for $1\leq k\leq d$.

\end{enumerate}
\end{theorem}
\begin{proof}
	Assume first that there exists a closed subspace $\fV\subset \fF^2_d$ which is co-invariant for $L$ and an injective operator $X:\fV\to \fH$ with dense range such that  $\|X\|\leq C$ and 
	\[
	XP_{\fV}L_k|_{\fV}=T_kX
	\]
	for $k=1,\dots,d$. Let $\xi=XP_\fV \Omega$. Since $\Omega$ is obviously cyclic for $L$ and $\fV$ is co-invariant, it is readily seen that $P_\fV \Omega$ is cyclic for $(P_{\fV}L_1|_{\fV},\ldots, P_{\fV}L_d|_{\fV})$.
	For a word $w\in \bF^+_d$, we compute
	\begin{align*}
		T_w \xi &=T_w XP_{\fV}\Omega = XP_{\fV}L_w P_{\fV}\Omega
	\end{align*}
	and thus conclude that $\xi$  is cyclic for $T$, since $X$ has dense range. In addition, for $w\in \bF_d^+$ and $h\in \fH$ we see that
	\begin{align*}
		\langle X^*h,  L_w\Omega\rangle&=\langle X^*h, P_\fV L_w \Omega\rangle =\langle X^*h, P_\fV L_w P_{\fV} \Omega\rangle\\
		&= \langle h, XP_\fV L_w P_{\fV} \Omega\rangle=\langle h, T_w \xi\rangle,
	\end{align*}
	and since $\{L_w\Omega\}_{w\in \bF_d^+}$ is an orthonormal basis for $\fF_d^2$, we infer that
	\begin{align*}
		\|X^* h\|^2 &= \sum_{w\in \bF_d^+}|\langle X^*h, L_w \Omega\rangle|^2 = \sum_{w\in \bF_d^+}|\langle h, T_w \xi\rangle|^2\\
							  &= \sum_{w\in \bF_d^+}\langle (T_w \xi\otimes T_w\xi)h,h\rangle
	\end{align*}
	whence
	\begin{equation}\label{Eq:CycCond}
		\sum_{w\in\bF_d^+} T_w \xi \otimes T_w \xi \leq C^2I.
	\end{equation}
	This shows that (iii) implies (i).
	Assume next that there is a cyclic vector $\xi\in \fH$ for $T$ satisfying \eqref{Eq:CycCond}.
	Then, given $h\in \fH$, we see that
	\begin{align*}
		\sum_{w\in \bF_d^+}|\langle h, T_w \xi\rangle|^2 = \sum_{w\in \bF_d^+}\langle (T_w \xi\otimes T_w\xi)h,h\rangle\leq C^2 \|h\|^2.
	\end{align*}
	We conclude that the vector
	\[
		\sum_{w\in \bF_d^+}\langle h, T_w \xi\rangle L_w\Omega
	\]
	belongs to $\fF_d^2$ and has norm at most $C\|h\|$.
	Hence, we may define a bounded linear operator $Y:\fH\to \fF_d^2$ by
	\[
		Yh = \sum_{w\in \bF_d^+}\langle h, T_w \xi\rangle L_w\Omega, \quad h\in \fH.
	\]
	Plainly $\|Y\|\leq C$.
	Moreover, given $1\leq k\leq d$, a word $u\in \bF_d^+$, and a vector $h\in \fH$ we find
	\begin{align*}
		\langle YT^*_k h, L_u\Omega\rangle&=\left\langle\sum_{w\in \bF_d^+}\langle T_k^*h, T_w \xi\rangle L_w\Omega, L_u \Omega\right\rangle=\langle T_k^*h, T_u \xi\rangle\\
		&=\langle h,T_k T_u \xi\rangle
	\end{align*}
	and
	\begin{align*}
		\langle L^*_k Yh, L_u\Omega\rangle&=\langle Yh, L_kL_u\Omega\rangle=\left\langle\sum_{w\in \bF_d^+}\langle h, T_w \xi\rangle L_w\Omega, L_kL_u \Omega\right\rangle\\
		&=\langle h,T_k T_u\xi\rangle.
	\end{align*}
	We conclude that $YT^*_k=L^*_k Y$ for every $1\leq k\leq d$. Furthermore, we note that if $Yh=0$ then
	\[
		0=\langle Yh,L_w\Omega \rangle=\langle h,T_w\xi\rangle
	\]
	for every $w\in \bF_d^+$, which implies that $h=0$ since $\xi$ is assumed to be cyclic for $T$. Thus, $Y$ is injective. Choosing $X=Y^*$ shows that (i) implies (ii). 
	
	Finally, the fact that (ii) implies (iii) follows from Lemma \ref{L:denserangeinj}.
\end{proof}

It is now a trivial matter to extract the desired necessary and sufficient condition.

\begin{corollary}\label{C:multonetransform}
Let $T=(T_1,\ldots,T_d)$ be a row contraction on some Hilbert space $\fH$. Then, the following statements are equivalent.

\begin{enumerate}

	\item[\rm{(i)}] There is a cyclic vector $\xi\in \fH$ for $T$ such that the operator $\sum_{w\in\bF_d^+} T_w \xi \otimes T_w \xi$ is bounded.

	\item[\rm{(ii)}] There exists a closed subspace $\fV\subset \fF^2_d$ which is co-invariant for $L$ and  such that $P_{\fV}L|_{\fV}\prec T.$
\end{enumerate}
\end{corollary}
\begin{proof}
This follows immediately from Theorem \ref{T:weakorth}.
\end{proof}

The property above that $\sum_{w\in\bF_d^+} T_w \xi \otimes T_w \xi$ be bounded can be seen as a relaxation of the condition that $\xi$ be a \emph{wandering} vector, in the sense that the set $\{T_w \xi\}_{w\in \bF^d_+}$ is orthogonal. It will turn out below that this weak orthogonality condition is automatically satisfied for pure row contractions that admit any cyclic vector (see Corollary \ref{C:cyclicweakorth}).

Given a $d$-tuple $T=(T_1,\ldots,T_d)$ on some Hilbert space $\fH$, we defined its \emph{multiplicity} to be the least cardinality $\mu_T$ of a cyclic set for $T$. We emphasize that in the special case of $d$-shifts, this definition of multiplicity agrees with the one we gave in Section \ref{S:prelim}. The following simple observation will be use implicitly.
Let $T'$ be a $d$-tuple of operators on some other Hilbert space $\fH'$ for which there exists an operator $X:\fH'\to \fH$ with dense range such that $XT'_k=T_kX$ for $k=1,\dots,d$. If $\Gamma\subset \fH'$ is cyclic for $T'$, then a straightforward verification reveals that $X\Gamma \subset \fH$ is cyclic for $T$. In particular, this shows that $\mu_T\leq \mu_{T'}$.

In preparation for our main result, we need a technical tool. 

\begin{lemma}\label{L:pureshiftmult}
	Let $T=(T_1,\ldots,T_d)$ be a pure row contraction on some Hilbert space $\fH$.
	Then, there is a $d$-shift $V=(V_1,\ldots,V_d)$ acting on some Hilbert space $\fK$  with multiplicity $\mu_T$ and a contractive operator $X:\fK\to \fH$ with dense range such that $T_kX=XV_k$ for $1\leq k\leq d$.
\end{lemma}
\begin{proof}
	By Theorem \ref{T:popescudilation}, there is a Hilbert space $\fK_0$ containing $\fH$ along with a $d$-shift $S=(S_1,\ldots,S_d)$ on $\fK_0$ with $\mu_S=\fd_T$ and such that
	\[
		T^*_k=S_k^*|{\fH}, \quad 1\leq k\leq d. 
	\]
By assumption, there is a subset $\Gamma\subset \fH$ of cardinality $\mu_T$ such that
	\[
		\fH=\ol{\spn\{ T_w \gamma: w\in \bF_d^+,\gamma\in \Gamma\}}.
	\]
Let 
	\[
		\fK=\ol{\spn\{ S_w \gamma: w\in \bF_d^+,\gamma\in \Gamma\}}
	\]
and put $V_k=S_k|_{\fK}$ for $k=1,\dots,d$. For each $n\in \bN$, we note that
	\[
		\sum_{|w|=n} V_w V_w^*=\sum_{|w|=n} P_{\fK}S_w P_{\fK}S^*_wP_{\fK}\leq P_{\fK}\left(\sum_{|w|=n} S_w S^*_w\right)P_{\fK}
	\]
which shows at once that $V$ is pure, and that $\sum_{k=1}^d V_kV_k^*\leq I$.
By the non-commutative Wold--von Neumann decomposition  \cite[Theorem 1.3]{popescu1989}, we infer that $V$ is a $d$-shift.

Next, define $X=P_{\fH}|_{\fK}:\fK\to \fH$.
For $k=1,\dots,d$, using that $\fH$ and $\fK$ are respectively co-invariant and invariant for $S$ we compute
\begin{align*}
	T_k X	&=	P_{\fH}S_kP_{\fH}|_{\fK} = P_{\fH}S_k|_{\fK}\\
				&=	P_{\fH}P_{\fK}S_k|_{\fK} = XS_k|_{\fK} = XV_k.
\end{align*}
Hence, for $w\in \bF^+_d$ and $\gamma\in \Gamma$ we find
\[
XS_w\gamma=XV_w\gamma=T_w X\gamma=T_w\gamma
\]
which implies $\ol{X\fK} =\fH$, so that $X$ has dense range. In particular, we must have that $\mu_T\leq \mu_V$. 
On the other hand, by definition of $\fK$ we see that $\Gamma$ is a cyclic set for $V$, so that $\mu_V\leq \card (\Gamma)=\mu_T$.  
\end{proof}

We now arrive at the last technical step before the main result of the section. Roughly speaking, it says that up to a quasi-affine transform, the multiplicity of a pure row contraction can be made to agree with its defect.

\begin{lemma}\label{L:qashift}
	Let $T=(T_1,\ldots,T_d)$ be a pure row contraction on some Hilbert space $\fH$.
	Then, there is a pure row contraction $T'=(T_1',\ldots,T_d')$ on some other Hilbert space $\fH'$ such that $\mu_{T'}=\fd_{T'}=\mu_T$, along with a contractive injective operator $X:\fH'\to \fH$ with dense range such that $T_kX=XT'_k$ for $1\leq k\leq d$.	If $T$ is commuting, then so is $T'$. 
\end{lemma}
\begin{proof}
	By Lemma \ref{L:pureshiftmult}, there a $d$-shift $V=(V_1,\ldots,V_d)$ with $\mu_V=\mu_T$ acting on some Hilbert space $\fK$ and a contractive operator $Y:\fK\to \fH$ with dense range such that $T_kY=YV_k$ for every $1\leq k\leq d$.  Put 
	\[
	\fH'=\fK\ominus \ker Y \qand	X=Y|_{\fH'}
	\]
	and for each $1\leq k\leq d$ let 
	\[
	T'_k=P_{\fH'}V_k|_{\fH'}.
	\]
	Clearly, $X$ is contractive, injective, and it has dense range. Moreover, it follows from Lemma \ref{L:denserangeinj} that $T_kX=XT'_k$ for $1\leq k\leq d$.  In particular
	\begin{equation}\label{E:ineq1}
	\mu_T\leq \mu_{T'}.
	\end{equation} 
	Since $\fH'$ is co-invariant subspace for $V$, we see that $T'^*_w=V^*_w|_{\fH'} $ for  every $w\in \bF^+_d$. It follows that $T'$ is a pure row contraction. 	Using that
	\begin{align*}
		I-\sum_{k=1}^d T_k'T_k^{'*} = P_{\fH'}\left(I-\sum_{k=1}^d V_kV_k^{*}\right)|_{\fH'} 
	\end{align*}
	we find $\fd_{T'}\leq \fd_{V}$. But $V$ is a $d$-shift, so $\fd_V=\mu_V$ by Theorem \ref{T:shiftmult} and we infer 
	\begin{equation}\label{E:ineq2}
	\fd_{T'}\leq \mu_T.
	\end{equation}  
	On the other hand, since $T'$ is pure we may apply Theorem \ref{T:popescudilation} to find a Hilbert space $\fK'$ containing $\fH'$ and a $d$-shift $S=(S_1,\ldots,S_d)$ with $\fd_{T'}=\mu_S$ and such that $T^{'*}_k=S^*_k|_{\fH'}$ for every $1\leq k\leq d.$ Choose a subset $\Gamma\subset \fK'$ with cardinality $\fd_{T'}$ that is cyclic for $S$.
	For $w\in \bF^+_d$ and $\gamma\in \Gamma$, we may use that $\fH'$ is co-invariant for $S$ to find
	\begin{align*}
		T'_wP_{\fH'}\gamma &= P_{\fH'} S_wP_{\fH'}\gamma = P_{\fH'} S_w\gamma.
	\end{align*}
	Consequently, we have that
	\begin{align*}
	\ol{\spn\{ T'_wP_{\fH'}\gamma:w\in \bF^+_d,\gamma\in \Gamma\}}&= P_{\fH'}\left(\ol{\spn\{ S_w\gamma:w\in \bF^+_d,\gamma\in \Gamma\}}\right)\\
	&=P_{\fH'}\fK'=\fH'
	\end{align*}
	so that $P_{\fH'}\Gamma$ is cyclic for $T'$. Hence, 
	\begin{equation}\label{E:ineq3}
	\mu_{T'}\leq \fd_{T'}.
	\end{equation}  
	Combining inequalities (\ref{E:ineq1}), (\ref{E:ineq2}) and  (\ref{E:ineq3}) yields  $\mu_{T'}=\fd_{T'}=\mu_T$ as desired.
	
	Finally, if $T$ is commuting, then $\Ann_{\nc}(T)$ contains the commutator ideal, so that $T'$ is also commuting by Lemma \ref{L:anntransf}.
	\end{proof}

We remark here that the preceding two proofs are faithful adaptations of the single-variable ones found in \cite[Lemma I.3.5 and Theorem I.3.7]{bercovici1988}. Interestingly however, this approach forces us through the non-commuting multivariate world, even if we start with a commuting row contraction. 

To see why, we note that there are examples of cyclic invariant subspaces for $M_x$ for which the restriction has \emph{infinite} defect. Take for instance the cyclic invariant subspace $\fM\subset H^2_d$ generated by $x_1$, so that
\[
\fM=\bigoplus_{\alpha\in \bN^d, \alpha_1\geq 1}  \bC x^\alpha.
\] 
Then $\fM$ clearly has infinite co-dimension, so we may invoke \cite[Theorem F]{arveson2000} to conclude that the restriction $M_x|_{\fM}$ has infinite defect. This underlines potential difficulties in establishing the equality $\fd_V=\mu_V$, which is used crucially in the proof of Lemma \ref{L:qashift}. Fortunately, with our approach this equality follows from Theorem \ref{T:shiftmult} since $V$ is known to be a $d$-shift, being the restriction of a $d$-shift to an invariant subspace.  For commuting $d$-shifts, the behaviour of such restrictions is more subtle. Assume for instance that the restriction of the standard commuting $d$-shift to some cyclic proper invariant subspace is another commuting $d$-shift. Since it must be cyclic, this second commuting $d$-shift must be of multiplicity one, and thus unitarily equivalent to the standard commuting $d$-shift. This is however impossible, by \cite[Corollary 5.5]{GHX04}.  Although it is plausible that a completely different approach could circumvent these issues in the commuting setting, the preceding remarks show that our current proofs have an inexorable non-commutative aspect built into them.

The following is one of our main results.

\begin{theorem}\label{T:cyclicpure}
	Let $T=(T_1,\ldots,T_d)$ be a pure cyclic row contraction on some Hilbert space $\fH$.
	Then there is a contractive operator $X:\fF_d^2\to \fH$ with dense range such that $XL_k=T_kX$ for $k=1,\dots,d$.
	
	If $T$ is commuting, then there is a contractive operator $X':H^2_d\to\fH$ with dense range such that $X'M_{x_k}=T_kX'$ for $k=1,\dots,d$.
\end{theorem}
\begin{proof}
	By Lemma \ref{L:qashift}, we see that there is a pure row contraction $T'=(T_1',\ldots,T_d')$ on some other Hilbert space $\fH'$ such that $\fd_{T'}=1$, along with a contractive injective operator $Y:\fH'\to \fH$ with dense range such that $T_kY=YT'_k$ for $1\leq k\leq d$. 	Up to unitary equivalence, we may assume by Theorem \ref{T:popescudilation} that $\fH'\subset \fF_d^2$ and 
	$T^{'*}_k = L^*_k|_{\fH'}$ for every $k=1,\dots,d.$ Define $X:\fF_d^2\to \fH$ as $X=YP_{\fH'}$.
	For $1\leq k\leq d$, we calculate using that $\fH'$ is co-invariant for $L$ that
	\begin{align*}
		XL_k &= YP_{\fH'}L_k = YP_{\fH'}L_kP_{\fH'}\\
				 &= YT'_kP_{\fH'} = T_kYP_{\fH'}\\
				 &= T_kX.
	\end{align*}
	Moreover, we note that
	\[
		\ol{X\fF_d^2}=\ol{YP_{\fH'}\fF_d^2}=\ol{Y\fH'}=\fH
	\]
	so that $X$ has dense range, which establishes the first statement. For the second, we assume in addition that $T$ is commuting. Then, 
	\[
	0=(T_jT_k-T_kT_j)X=X(L_jL_k-L_kL_j)
	\]
	and thus $\ker X$ contains $(L_j L_k -L_k L_j)\fF^2_d$	for every $1\leq j,k\leq d$.
	We infer that $\ker X$ contains $(\fF^s_d)^\perp$ by \cite[Proposition 2.4]{davidson1998alg}, whence $X=XP_{\fF^s_d}$.
	By Theorem \ref{T:quotientcomm}, there is a unitary operator $U:H^2_d\to \fF^s_d$ with the property that
\[
UM_{x_k}U^*=P_{\fF^s_d}L_k|_{\fF^s_d}, \quad 1\leq k\leq d.
\]
	We set $X'=XU$, which still has dense range. Using that $\fF^s_d$ is co-invariant for $L$ we obtain
	\begin{align*}
	X'M_{x_k}&=XUM_{x_k}=XP_{\fF^s_d}L_kU\\
	&=XL_kU=T_kX' 
	\end{align*}
	for $k=1,\dots,d$.
\end{proof}

Comparing the previous result with Theorem \ref{T:weakorth}, we obtain the following consequences.

\begin{corollary}\label{C:cyclicweakorth}
	Let $T=(T_1,\ldots,T_d)$ be a pure cyclic row contraction on some Hilbert space $\fH$. Then, there is a vector $\xi\in \fH$ that is cyclic for $T$ with the additional property that $\sum_{w\in\bF_d^+} T_w \xi \otimes T_w \xi\leq I.$ Moreover, the following statements hold.
	\begin{enumerate}
	
	\item[\rm{(1)}] If $\|T_w\xi\|=1$ for every $w\in \bF_d^+$, then $T$ must be unitarily equivalent to the standard $d$-shift.

	\item[\rm{(2)}] If $d>1$ and $T$ is commuting, then $\|T_w\xi\|<1$ for some $w\in\bF_d^+$.
	\end{enumerate}
\end{corollary}
\begin{proof}
	The existence of a vector $\xi$ with the announced property follows by combining Theorems \ref{T:weakorth} and \ref{T:cyclicpure}. For (1), assume that $\|T_w\xi\|=1$ for each $w\in \bF_d^+$. The condition that $\sum_{w\in\bF_d^+} T_w \xi \otimes T_w \xi\leq I$ then forces $\{T_w\xi:w\in \bF_d^+\}$ to be an orthonormal set.
	But since $\xi$ is cyclic, the set $\{T_w\xi:w\in \bF_d^+\}$ must in fact be an orthonormal basis for $\fH$. It is now straightforward to construct a unitary equivalence between $T$ and the standard $d$-shift $L$. Finally, (2) is trivial consequence of (1) since the standard $d$-shift is non-commuting.
\end{proof}

We close this section by refining Theorem \ref{T:cyclicpure} to obtain another one of our main results. See Subsection \ref{SS:functcalc} for the definition of the $d$-tuple $M_\J$.

\begin{corollary}\label{C:cyclicann}
	Let $T=(T_1,\ldots,T_d)$ be a pure cyclic commuting row contraction on some Hilbert space $\fH$.
	Let  $\J=\Ann_{\comm}(T)$. Then $M_\J$ is a quasi-affine transform of $T$.
\end{corollary}
\begin{proof}

	By Theorem \ref{T:cyclicpure},  there is a contractive operator $X:H^2_d\to\fH$ with dense range such that $XM_{x_k}=T_kX$ for $k=1,\dots,d$. In turn, invoking Lemma \ref{L:denserangeinj}, we see that there is a closed subspace $\fK\subset H^2_d$ which is co-invariant for $M_x$ and such that
	\[
	(P_{\fK}M_{x_1}|_{\fK},\ldots,P_{\fK}M_{x_d}|_{\fK})\prec T.
	\]
	It follows from Lemma \ref{L:anntransf}   that
	\[
	\Ann_{\nc}(T)=\Ann_{\nc}(P_{\fK}M_{x_1}|_{\fK},\ldots,P_{\fK}M_{x_d}|_{\fK})
	\]
	whence 
	\[
	\Ann_{\comm}(T)=\Ann_{\comm}(P_{\fK}M_{x_1}|_{\fK},\ldots,P_{\fK}M_{x_d}|_{\fK})
	\]
	by virtue of Equation (\ref{Eq:ann}). By Lemma \ref{L:anncoinvcomm}, we conclude that $\fK=H_\J$ so that $M_\J\prec T$.	
	
\end{proof}

Given  a pure cyclic row contraction $T=(T_1,\ldots,T_d)$ with $\fb=\Ann_{\nc}(T)$, we do not claim that $L_\fb\prec T$.  The difficulty in achieving this is connected with the discussion following Lemma \ref{L:anncoinv}. This is also related to determining whether a given co-invariant subspace coincides with the so-called ``maximal $\fb$-constrained piece" of $L$ as defined in \cite{popescu2006}. We will not address these issues further in this paper.

To summarize, the basic outcome of this section is a classification of some pure row contractions by means of compressions of the standard $d$-shift. By allowing for quasi-affine transforms we capture the behaviour of a class of objects that is more flexible than those with defect equal to one (which is required in Theorem \ref{T:popescudilation}). To illustrate this flexibility,  we note for instance that if $T$ is a cyclic pure row contraction with defect one, then the scaled row contraction $rT$ has defect equal to $\dim \fH$ for $0<r<1$, even though it is still pure and cyclic. Furthermore, the cyclicity condition is preserved under similarity.

\section{Rigidity of invariant subspaces}\label{S:rigidity}

In this section, we examine the rigidity of invariant or co-invariant subspaces for pure row contractions. More precisely, assume that $T=(T_1,\ldots,T_d)$ is a pure row contraction on some Hilbert space $\fH$, and let $\fM\subset \fH$ be a closed subspace which is invariant for $T$. Throughout, we let $T|_\fM$ denote the restricted $d$-tuple
\[
(T_1|_{\fM},\ldots,T_d|_{\fM})
\]
and $P_{\fM^\perp}T|_{\fM^\perp}$ denote the compressed $d$-tuple
\[
(P_{\fM^\perp}T_1|_{\fM^\perp},\ldots,P_{\fM^\perp}T_d|_{\fM^\perp}).
\]
We aim to determine the extent to which the annihilators 
\[
\Ann_{\nc}(T|_\fM) \qor \Ann_{\nc}(P_{\fM^\perp}T|_{\fM^\perp})
\] 
determine $\fM$. We begin by revisiting a pathological example that we encountered before.

\begin{example}\label{E:kribs}
By \cite[Theorem 3.7]{kribs2001Ld}, there is a proper closed subspace $\fK\subset \fF^2_d$ which is co-invariant for $\fL_d$ and which satisfies
\[
\{\theta\in \fL_d:\theta \fF^2_d\subset \fK^\perp\}=\{0\}.
\]
In particular, we see that 
\[
\Ann_{\nc}(P_{\fK}L|_{\fK})=\{0\}=\Ann_{\nc}(L)
\]
yet $\fK\neq \fF^2_d$. \qed
\end{example}

In view of this difficulty, we will focus our attention for the remainder of the paper on the commuting setting. 
We start by considering a relatively simple case. A commuting $d$-tuple $(N_1,\ldots,N_d)$ is said to be \emph{nilpotent} if for each $1\leq k\leq d$ there is $m_k\in \bN$ such that $N_k^{m_k}=0$. 
\begin{theorem}
	Let $N=(N_1,\dots,N_d)$ be a cyclic nilpotent commuting row contraction on some Hilbert space $\fH$. Let $\fM\subset \fH$ be an invariant subspace for $N$ such that $\Ann_{\comm}(N|_\fM)=\Ann_{\comm}(N)$. Then, $\fM=\fH$.
	\label{T:NilpNoPropSub}
\end{theorem}

\begin{proof}
	Let $\xi\in \fH$ be a cyclic vector for $N$. We find
	\[
	\fH=\ol{\spn\{N^\alpha\xi:\alpha\in \bN^d\}}.
	\]
	Since $N$ is nilpotent, this implies that $\fH$ is finite-dimensional and 
	\[
	\fH=\{p(N)\xi:p\in \bC[x_1,\ldots,x_d]\}.
	\]
	The subset $\J\subset \bC[z_1,\ldots,z_d]$ defined as
	\[
	\J=\{p\in \bC[z_1,\ldots,z_d]:p(N)\xi\in \fM\}
	\]
	is an ideal with the property that
	\[
	\fM=\{p(N)\xi:p\in \J\}.
	\]
	For convenience, we set
	\[
	\A=\Ann_{\comm}(N|_\fM)\cap \bC[x_1,\ldots,x_d]
	\]
	and
	\[
	\B=\Ann_{\comm}(N)\cap \bC[x_1,\ldots,x_d].
	\]
	Observe that a polynomial $p$ satisfies $p(N)=0$ if and only if $p(N)\xi=0$. Thus,
	\[
	\A=\{f\in \bC[x_1,\ldots,x_d]: f\J\subset \B \}.
	\]
	Clearly, the set $\{\alpha\in \bN^d: x^\alpha\notin \B\}$ is finite and contains $(0,\ldots,0)$. Choose an element $\beta$ of that set with maximal length. Then, we have $x_k x^\beta\in \B$ for every $1\leq k\leq d$.

	Assume now that $\fM$ is a proper subspace of $\fH$, whence $\xi\notin \fM$. Let $q$ be a polynomial with non-zero constant term. Then, there is another polynomial $r$ such that $r(N)q(N)=I$, so in particular $\xi=r(N)q(N)\xi$. We conclude that $q\notin \J$, which shows $\J$ consists of polynomials with zero constant term. Hence, $x^\beta \J\subset \B$, and therefore $x^\beta\in \A$. This shows that $\A\neq \B$ and in particular $\Ann_{\comm}(N|_\fM)\neq \Ann_{\comm}(N)$.
\end{proof}

In the case of a single cyclic constrained contraction, the conclusion of the previous theorem always holds \cite[Theorem III.2.13]{bercovici1988}. Unfortunately, the general multivariate situation is more complicated. In fact, the aforementioned theorem does not even extend to the commuting bivariate case, as the following example shows. 

\begin{example}\label{E:invsubann}
  Let $\J$ be the weak-$*$ closed ideal of $\M_2$ generated by $x_2$. Then, the space $H_\J=H^2_2\ominus [\J H^2_2]$ consists of those functions in $H^2_2$ that depends only on the variable $x_1$ and in particular we have
  \[
  P_{H_\J}M_{x_2}|_{H_\J}=0.
  \]
The pure commuting pair 
  \[
  M_\J=(P_{H_\J}M_{x_1}|_{H_\J},0)
  \]
is constrained, and by Lemma \ref{L:uecomm} we have $\Ann_{\comm}(M_\J)=\J$. Note also that $1\in H_\J$ is a cyclic vector for $M_\J$. Let $U:H^2_1\to H_\J$ be the linear operator defined as
\[
(Uf)(z_1,z_2)=f(z_1), \quad (z_1,z_2)\in\bB_2
\]
for $f\in H^2_1$. It is readily verified that $U$ is unitary. Observe now that 
\[
U^*(P_{H_\J}M_{x_1}|_{H_\J})U
\] 
is the usual isometric $1$-shift.	Let $\theta\in\M_1$ be a non-constant inner function and let $\fM=U(\theta H^2_1)\subset H_\J$. Then, $\fM$ is a proper closed subspace of $H_\J$ which is invariant for $M_\J$. If we let $R_1=P_{H_\J}M_{x_1}|_{\fM}$, then we see that $(R_1,0)$ is the restriction of $M_\J$ to $\fM$. Further, if we let $U'=UM_{\theta}U^*$ then we see that $U':H_\J\to\frk{M}$
	 is unitary with
	\[
	U'(P_{H_\J}M_{x_1}|_{H_\J})U'^*=R_1.
	\]
	Thus $R$ is unitarily equivalent to $M_\J$, and in particular it is a pure cyclic commuting row contraction with
	\[
	\Ann_{\nc}(R)=\Ann_{\nc}(M_\J) \qand \Ann_{\comm}(R)=\Ann_{\comm}(M_\J),
	\]
	even though $\fM\neq H_\J$.
	
\qed
\end{example}

We note that the zero set of the commutative annihilator has dimension one in the preceding example. At present, we know of no multivariate counterexample to \cite[Theorem III.2.13]{bercovici1988} where this zero set has dimension zero. Theorem \ref{T:NilpNoPropSub} may be evidence towards the relevance of having a zero-dimensional zero set, but at the time of this writing we do not know for sure.

In light of Example \ref{E:invsubann}, we then seek natural sufficient conditions for multivariate rigidity of invariant subspaces. The next development will gain insight by first considering co-invariant subspaces; it is based on the following simple fact.

\begin{lemma}\label{L:rigidcoinv}
	Let $\fM,\fN\subset H^2_d$ be closed subspaces that are co-invariant for $\M_d$.
	Then, we have $\fM=\fN$ if and only if $\Ann_{\comm}(P_{\fM}M_x|_\fM)=\Ann_{\comm}(P_{\fN}M_x|_\fN)$.
\end{lemma}
\begin{proof}
This follows immediately from Lemma \ref{L:anncoinvcomm}.
\end{proof}

Before proceeding, we introduce some notation. Let  $S=(S_1,\ldots,S_d)$ be a $d$-tuple of operators on some Hilbert space $\fH$ and let $S'=(S'_1,\ldots,S'_d)$ be another $d$-tuple of operators on some other Hilbert space $\fH'$. Denote by $\Q(S',S)$ the collection of injective bounded linear operators $X :\fH'\to \fH$ with dense range such that $X S'_k=S_kX $ for every $k=1,\dots,d$. 
Recall that if $T$ is a cyclic pure commuting row contraction and $\J=\Ann_{\comm}(T)$, then $\Q(M_\J,T)$ is non-empty by Corollary \ref{C:cyclicann}.

We now prove a rigidity result for co-invariant subspaces of cyclic pure commuting row contractions.

\begin{theorem}\label{T:rigidityQcoinv}
	Let $T=(T_1,\ldots,T_d)$ be a cyclic pure commuting row contraction on some Hilbert space $\fH$. Put $\J=\Ann_{\comm}(T)$.
	Let $\fM,\fN\subset \fH$ be closed co-invariant subspaces for $T$ such that
	\[
		\Ann_{\comm}(P_{\fM}T|_{\fM})=\Ann_{\comm}(P_{\fN}T|_{\fN}).
	\]
	Then, $\ol{X ^*\fM}=\ol{X ^*\fN}$ for every $X \in \Q(M_\J,T)$.
	In particular, $\fM=\fN$ if $X ^*\fM$ and $X ^*\fN$ are closed for some $X \in \Q(M_\J,T)$.
\end{theorem}

\begin{proof}
	Fix $X \in \Q(M_\J,T)$. Define $\fM'=\ol{X ^*\fM}$ and $\fN'=\ol{X ^*\fN}$, so that $\fM'$ and $\fN'$ are co-invariant for $M_\J$, and hence for $M_x$. Let 
	$
		Z=P_{\fM}X |_{\fM'}:\fM'\to \fM.
	$
	We have $Z^*=X ^*|_{\fM}$, so $Z^*$ is injective with dense range and 
	\[
	Z^*\in \Q(T^*|_\fM,M_x^*|_{\fM'}).
	\]
	Thus,
	\[
	Z\in \Q(P_{\fM'}M_x|_{\fM'},P_\fM T|_{\fM}).
	\]
	Likewise, the operator $Y=P_{\fN}X |_{\fN'}$
	is injective with dense range, and
	\[
	Y\in\Q(P_{\fN'}M_x|_{\fN'},P_\fN T|_{\fN}).
	\]
	We conclude from Lemma \ref{L:anntransf} and from the assumption on $\fM$ and $\fN$ that
	\[
	\Ann_{\comm}(P_{\fM'}M_x|_{\fM'})=\Ann_{\comm}(P_{\fN'}M_x|_{\fN'}).
	\]
	By virtue of Lemma \ref{L:rigidcoinv}, we obtain $\fM'=\fN'$, that is  $\ol{X ^*\fM}=\ol{X ^*\fN}$.
	If $X ^*\fM$ and $X ^*\fN$ are closed, then $X ^*\fM=X ^*\fN$ and we conclude that $\fM=\fN$ since $X ^*$ is injective.
\end{proof}

The following consequence is noteworthy as it applies in particular to the nilpotent setting, thus complementing Theorem \ref{T:NilpNoPropSub}.

\begin{corollary}\label{C:rigidity1}
	Let $T=(T_1,\ldots,T_d)$ be a cyclic pure commuting row contraction on some finite-dimensional Hilbert space $\fH$.
	Let $\fM,\fN\subset \fH$ be co-invariant subspaces for $T$ such that
	\[
		\Ann_{\comm}(P_{\fM}T|_\fM)=\Ann_{\comm}(P_{\fN}T|_\fN).
	\]
	Then, $\fM=\fN$.
\end{corollary}
\begin{proof}
Finite-dimensional subspaces are closed, so this follows at once from Theorem \ref{T:rigidityQcoinv}.
\end{proof}

We now return to invariant subspaces and identify a sufficient condition for rigidity.

\begin{theorem}\label{T:rigidityQinv}
	Let $T=(T_1,\ldots,T_d)$ be a pure commuting row contraction on some Hilbert space $\fH$. Let $\J\subset \M_d$ be a weak-$*$ closed ideal and let $X \in \Q(M_\J^*,T)$. Let $\fM,\fN\subset \fH$ be closed invariant subspaces for $T$ such that $\Ann_{\comm}(T|_\fM)=\Ann_{\comm}(T|_\fN)$. If $\fM$ and $\fN$ are contained in the range of $X$, then $\fM=\fN$.
\end{theorem}
\begin{proof}
If $f\in X^{-1}\fM$, then for every $1\leq k\leq d$ we obtain
	\[
		X (M^*_{x_k}|_{H_\J}) f=T_k X f\in T_k \fM\subset \fM.
	\]
	Thus $X^{-1}\fM$ is invariant for $M_\J^*$ and hence for $M_x^*$. An identical argument shows that $X^{-1}\fN$ is also invariant for $M_x^*$.
	Next, let 
	\[
		Z=X|_{X^{-1}\fM}:X^{-1}\fM\to \fM
	\]
	\[
	 Y=X|_{X^{-1}\fN}:X^{-1}\fN\to \fN.
	\]
	Because $X$ is injective, it follows that $Z$ and $Y$ are injective as well. Moreover, $Z$ and $Y$ are surjective since $\fM,\fN$ are contained in the range of $X$.
	For $k=1,\dots,d$, we note that
	\[
		(T_k|_{\fM})Z=Z(M_{x_k}^*|_{X^{-1}\fM}) \qand (T_k|_{\fN})Y=Y(M_{x_k}^*|_{X^{-1}\fN}).
	\]
	Let $J:H^2_d\to H^2_d$ be the conjugate linear, anti-unitary operator such that $Jx^\alpha=x^\alpha$ for every $\alpha\in \bN^d$. Then, we note that
	\[
		(T_k|_{\fM})Z=Z((JM_{x_k}J)^*|_{X^{-1}\fM}) \qand (T_k|_{\fN})Y=Y((JM_{x_k}J)^*|_{X^{-1}\fN})
	\]
	for every $k=1,\ldots,d$. 	A standard approximation procedure  yields that $J\M_d J=\M_d$ and
	\[
		\phi(T|_{\fM})Z=Z((JM_\phi J)^*|_{X^{-1}\fM}) \qand \phi(T|_{\fN})Y=Y((JM_\phi J)^*|_{X^{-1}\fN})
	\]
	for every $\phi\in \M_d$. Observe then that $\phi\in \Ann_{\comm}(T|_\fM)$ if and only if 
	\[
	P_{X^{-1}\fM}(JM_\phi J)|_{X^{-1}\fM}=0
	\]
	which in turn is equivalent to
	\[
	JM_\phi J\in \Ann_{\comm}(P_{X^{-1}\fM}M_x|_{X^{-1}\fM}).
	\]
	Thus, we find
	\[
	\Ann_{\comm}(P_{X^{-1}\fM}M_x|_{X^{-1}\fM})=J \Ann_{\comm}(T|_\fM) J.
	\]
	Likewise, we infer that 
	\[
	\Ann_{\comm}(P_{X^{-1}\fN}M_x|_{X^{-1}\fN})=J \Ann_{\comm}(T|_\fN) J.
	\]
	By assumption, we may therefore write
	\[
	\Ann_{\comm}(P_{X^{-1}\fM}M_x|_{X^{-1}\fM})=\Ann_{\comm}(P_{X^{-1}\fN}M_x|_{X^{-1}\fN}).
	\]
	In view of Lemma \ref{L:rigidcoinv}, we then know that $X^{-1}\fM=X^{-1}\fN$.
	Applying $X$ on both sides yields $\fM=\fN$.
\end{proof}

In general, it is not clear how to determine that $\Q(M_\J^*,T)$ is non-empty in order to apply the previous result. In some cases however, the existence of a cyclic vector can be exploited. 

\begin{corollary}\label{C:rigidity2}
	Let $T=(T_1,\ldots,T_d)$ be a nilpotent commuting row contraction on some Hilbert space $\fH$. Assume that $T^*$ is cyclic. Let $\fM,\fN\subset \fH$ be invariant subspaces for $T$ such that $\Ann_{\comm}(T|_\fM)=\Ann_{\comm}(T|_\fN)$. Then, $\fM=\fN$.
\end{corollary}
\begin{proof}
	The commuting $d$-tuple $T^*$ is also nilpotent.
	Applying \cite[Theorem 3.8]{popescu2014sim} with $k=m=1$, $\Q=\{0\}$ and $f(Z)=Z_1+\ldots+Z_d$ (in the notation of that paper), we see that there is an invertible operator $Y\in B(\fH)$ such that the $d$-tuple
	\[
	YT^{*}Y^{-1}=(YT_1^{*}Y^{-1},\ldots,YT_d^{*}Y^{-1})
	\]
	is a nilpotent cyclic commuting row contraction. In particular, $YT^{*}Y^{-1}$ is pure, so if we put $\J=\Ann_{\comm}(YT^{*}Y^{-1})$, then  by virtue of Corollary \ref{C:cyclicann}, we see that there is $X \in \Q(M_\J, YT^{*}Y^{-1})$.
	Next, observe that $\fH$ is finite-dimensional since the nilpotent $d$-tuple $T^*$ is cyclic, whence $X $ is invertible and $(X ^{-1}Y)^*\in \Q(M_\J^*,T)$. Theorem \ref{T:rigidityQinv} then yields $\fM=\fN$.
\end{proof}

It is not clear that Theorem \ref{T:NilpNoPropSub} could be derived from the previous result, as it is not typically true that the adjoint of a cyclic nilpotent row contraction is still cyclic; see Example \ref{E:maxcount} below.

\section{Invariant decompositions and separating vectors}\label{S:cyclicdecomp}

We saw in Corollary \ref{C:cyclicann} that pure commuting row contractions can be effectively classified using compressions of the standard commuting $d$-shift, provided that they admit a cyclic vector. Inspired by the successful univariate theory developed in \cite{bercovici1988}, a natural subsequent step in this classification program would involve moving past the setting of cyclicity using the following device. 

Let $T=(T_1,\ldots,T_d)$ be a $d$-tuple of operators on some Hilbert space $\fH$ and let $\fM,\fN\subset \fH$ be non-trivial invariant subspaces for $T$. We say that the pair $(\fM,\fN)$ is an \emph{invariant decomposition} for $T$ if $\fM\cap \fN=\{0\}$ and $\ol{\fM+\fN}=\fH$. In the special case where $T|_{\fM}$ is cyclic, we say that $(\fM,\fN)$ is a \emph{cyclic decomposition} for $T$.

In the single-variable case, non-cyclic constrained contractions always admit cyclic decompositions \cite[Theorem III.3.1]{bercovici1988}. We start this section with an example illustrating that cyclic decompositions are more elusive in several variables. For this purpose, we introduce one more notion. An invariant decomposition $(\fM,\fN)$ will  be said to be \emph{topological} if the stronger condition $\fM+\fN=\fH$ is satisfied. The following is an adaptation of \cite[Theorem 2.3]{griffith1969}

\begin{example}\label{E:FromGriff}
	Let $\{e_n:n\in \bN\}$ denote the canonical orthonormal basis of $\ell^2(\bN)$. Let $\fH=\ell^2(\bN)\oplus \ell^2(\bN)$. For each $n\in \bN$, define vectors  $\xi_n,\eta_n\in\fH$ as
	\[ \xi_n=e_n\oplus 0, \quad \eta_n=0\oplus e_n. \]
	Then, $\{\xi_n,\eta_n:n\in \bN\}$ is an orthonormal basis for $\fH$. We may define bounded linear operators $T_1,T_2$ on $\fH$ as
	\[
	T_1=\sum_{n=1}^\infty \frac{1}{\sqrt{2}}\eta_{n+1}\otimes \xi_n,
	\]
	\[
	T_2=\sum_{n=1}^\infty \frac{1}{\sqrt{2}}\eta_n\otimes \xi_n
	\]
	where the series converge in the strong operator topology of $B(\fH)$. 
	We find
	\[
	2T_1T_1^*=\sum_{n=1}^\infty \eta_{n+1}\otimes \eta_{n+1}
	\]
	and
	\[
	2T_2T_2^*=\sum_{n=1}^\infty \eta_{n}\otimes \eta_{n}.
	\]
	Thus, $2T_1T_1^*$ and $2T_2T_2^*$ are orthogonal projections. In particular, we see that $T_1$ and $T_2$ are scalar multiples of partial isometries, and thus have closed range. Moreover,
	\[
	T_1T_1^*+T_2T_2^*\leq I
	\]
	so that $T=(T_1,T_2)$ is a row contraction on $\fH$. Since $\xi_n$ is orthogonal to $\eta_m$ for every $n,m\in \bN$, we see that 
	\[
	T_1^2=T_1T_2=T_2T_1=T_2^2=0.
	\]
	Thus, $T$ is a nilpotent commuting row contraction. Using Theorem \ref{T:gleasontrick} it is readily seen that $\Ann_{\comm}(T)$ is the weak-$*$ closed ideal of $\M_2$ generated by $\{x_1^2,x_1x_2,x_2^2\}$. We infer that the range of $ T^\beta$ is closed for all $\beta\in\bN^2$.

	Assume that $\fM,\fN\subset \fH$ are closed invariant subspaces for $T$ with $\fM\cap \fN=\{0\}$ and $\fH=\fM+\fN$. 
	Observe that 
	\[
	T_1\fH=\oplus_{n=2}^\infty \bC \eta_n
	\]
	and
	\[
	T_2\fH=\oplus_{n=1}^\infty \bC \eta_n.
	\]
	Furthermore, we have that 
	\[
	T_1\fH=T_1\fM+T_1\fN, \quad T_2\fH=T_2\fM+T_2\fN
	\]
	and
	\[
	T_1\fM\subset \fM, \quad T_1 \fN\subset \fN,
	\]
	\[
	T_2\fM\subset \fM, \quad T_2 \fN\subset \fN.
	\] 
	Using that $\fM\cap \fN=\{0\}$ and that $T_1\fH\subset T_2\fH$, we infer that  $T_1\fM\subset T_2\fM$ and $T_1\fN\subset T_2\fN$. Since $T_1\fH$ is a co-dimension one closed subspace of $T_2\fH$, we infer that either $T_1\fM=T_2\fM$ or $T_1\fN=T_2\fN$. By symmetry, we may suppose the former.

	Given a subspace $\fK\subset \fH$, we set
	\[ \nu_\xi(\fK) = \inf\{n\in\bN: \xi_n\notin\fK^\perp\} \]
	and
	\[ \nu_\eta(\fK) = \inf\{n\in\bN: \eta_n\notin \fK^\perp\}.\]
	Note that $\nu_\xi(\fK)$ is infinite whenever $\xi_n\in \fK^\perp$ for every $n\in \bN$. Using that $T_1^*\eta_{n+1}=\xi_n$ and $T_2^*\eta_n=\xi_n$ for every $n\in \bN$, we see that 
	\[
	\nu_\eta(T_2\fK)=\nu_\xi(\fK) \qand \nu_\eta(T_1\fK)= \nu_\xi(\fK)+1.
	\]
	Recall that $T_1\fM=T_2\fM$, thus
	\[
	\nu_\xi(\fM)=\nu_\eta(T_2\fM)=\nu_\eta(T_1\fM)=\nu_\xi(\fM)+1
	\]
	which forces $\nu_\xi(\fM)=+\infty$. Therefore, $\xi_n\in \fM^\perp$ for every $n\in \bN$.
	Consequently, we find
	\[
	\fM\subset \oplus_{n=1}^\infty \bC \eta_n=T_2\fH.
	\]
	We infer that $\fM\subset T_2\fM$ which further implies that $T_2\fM=\fM$.
	Finally, using that $T^2_2=0$ we obtain $\fM=\{0\}$. This shows that $T$ does not admit a topological invariant decomposition. 
	
	If $(\fM,\fN)$ is a cyclic decomposition for $T$, then using that $T$ is nilpotent and that $T|_\fM$ is cyclic we conclude that $\fM$ must be finite-dimensional. It is well-known then that $\fH=\fM+\fN$, so that $(\fM,\fN)$ would be a topological invariant decomposition for $T$. We thus conclude that $T$ admits no cyclic decomposition either.
	\qed
\end{example}

We mention here a classical result of Apostol and Stampfli \cite[Theorem 5]{AS1976} which implies that if $T$ is a nilpotent operator on a Hilbert space $\fH$ with the property that $T^n \fH$ is closed for every $n\in \bN$, then $T$ admits a topological cyclic decomposition.  The previous example shows that this theorem fails in several variables.

In spite of Example \ref{E:FromGriff}, by leveraging the work done previously on rigidity of invariant subspaces we identify a sufficient condition that allows for an invariant decomposition.

\begin{theorem}\label{T:splitting}
	Let $T=(T_1,\ldots,T_d)$ be a nilpotent commuting row contraction on some Hilbert space $\fH$.
	Let $\fM\subset \fH$ be an invariant subspace for $T$ such that $(T|_{\fM})^*$ has a cyclic vector and $\Ann_{\comm}(T|_\fM)=\Ann_{\comm} (T)$.
	Then, there is a subspace $\fN\subset \fH$ that is invariant for $T$ such that $\fM\cap \fN=\{0\}$ and $\fH=\ol{\fM+\fN}$.
\end{theorem}
\begin{proof}
	By assumption, there is a vector $\xi\in \fM$ which is cyclic for $(T|_\fM)^*$. Let 
	\[
		\fK=\ol{\spn\{T^{*\alpha}\xi:\alpha\in \bN^d\}}.
	\]
	Put $S=P_{\fK} T|_{\fK}$ and observe that $\Ann_{\comm}(T)\subset \Ann_{\comm}(S)$. Define $X:\fM\to \fK$ as $X=P_{\fK}|_{\fM}$. Using that $\fK$ is coinvariant for $T$ and $\fM$ is invariant for $T$, we obtain
	\begin{align*}
		S_kX&= P_{\fK}T_kP_{\fK}|_{\fM}=P_{\fK}T_k|_{\fM}\\
		&=P_{\fK}P_{\fM}T_k|_{\fM}=X T_k|_{\fM}
	\end{align*}
	for every $1\leq k\leq d$, so that for $\alpha\in \bN^d$ we have
	\[
		X^*T^{*\alpha}|_{\fK}=X^*S^{*\alpha}=(T|_\fM)^{*\alpha} X^*.
	\]
	Furthemore, 
	\[
	X^*\xi=P_{\fM}\xi=\xi
	\]
	whence
	\begin{align*}
		 X^*T^{*\alpha}\xi=(T|_\fM)^{*\alpha} X^*\xi=(T|_\fM)^{*\alpha} \xi
	\end{align*}
	for every $\alpha\in \bN^d$.
	By choice of $\xi$ being cyclic for $(T|_\fM)^*$, we conclude that $X^*$ has dense range and thus $X$ is injective.
	Now, $X\fM\subset \fK$ is easily seen to be invariant for $S$ and
	\[
	(S_k|_{X \fM})X=X T_k|_{\fM}, \quad 1\leq k\leq d
	\]
	so it follows from Lemma \ref{L:anntransf} that
	\[
		\Ann_{\comm}(S|_{X\fM})\subset \Ann_{\comm}(T|_{\fM}).
	\]
	By assumption, we then infer that
	\[
	\Ann_{\comm}(S|_{X\fM})\subset\Ann_{\comm} (T)
	\]
	and thus
	\[
	\Ann_{\comm}(S|_{X\fM})\subset\Ann_{\comm}(S).
	\]
	We conclude that 
	\[
	\Ann_{\comm}(S|_{X\fM})=\Ann_{\comm} (S).
	\]
	Notice now that $S^*$ is cyclic by the construction of $\fK$.
	We may thus invoke Corollary \ref{C:rigidity2} to find $X\fM=\fK$, so that $X^*$ is injective.
	Finally, define $\fN=\fK^\perp$. We find
	\[
		\fM\cap \fN=\fM\cap \fK^\perp=\ker X=\{0\}
	\]
	and
	\[
		\ol{\fM+\fN}=(\fM^\perp \cap \fN^\perp)^\perp=(\fM^\perp\cap \fK)^\perp=(\ker X^*)^\perp=\fH. \qedhere
	\]
\end{proof}

One shortcoming of the previous theorem is that the existence of the decomposition  hinges on the \emph{adjoint} of $T|_\fM$ being cyclic. This condition is a bit mysterious.  Interestingly, in the case of a single constrained contraction, this condition is in fact equivalent to $T|_\fM$ being cyclic by \cite[Theorem III.2.3]{bercovici1988}. Unfortunately, this convenient equivalence does not hold in the multivariate world, as we will see in Example \ref{E:maxcount}.

Given these multivariate obstacles, and in light of  the univariate mechanism  \cite[Theorem III.3.1]{bercovici1988} for producing cyclic decompositions, it seems worthwhile to investigate more precisely \emph{where} this mechanism fails in several variables. The rest of the section will be devoted to this goal, through the lens of the following notion.

Let $T=(T_1,\ldots,T_d)$ be a constrained commuting row contraction on some Hilbert space $\fH$. A subset $\Sigma\subset \fH$ is said to be \emph{separating} for $T$ if, given any $\theta\in \M_d$, the  condition $\Sigma\subset \ker \theta(T)$ forces $\theta(T)=0$. This means that $\Sigma$ is a separating set for the operator algebra 
\[
\{\theta(T):\theta\in \M_d\}.
\]
When $\Sigma$ consists of only one element $\xi\in \fH$, we say that $\xi$ is a \emph{separating vector} for $\fH$. 
In the classical one-variable setting, separating vectors are called \emph{maximal} \cite{bercovici1988}, but we feel that in our context our choice of terminology is a bit more descriptive.

We start with a concrete example where a separating vector can be constructed explicitly. 

\begin{example}\label{Ex:Rectangle}
	Let $T=(T_1,\dots,T_d)$ be a pure commuting row contraction on some Hilbert space $\fH$. Assume that $\Ann_{\comm}(T)$ is the weak-$*$ closed ideal of $\M_d$ generated by $\{x_1^{n_1},\ldots,x_d^{n_d}\}$ for some fixed positive integers $n_1,\ldots,n_d$. 	Set
	\[ \Omega=\{\alpha\in\bN^d:x^\alpha\nin \Ann_{\comm}(T)\} \]
	and put
	\[
	\gamma=(n_1-1,\dots,n_d-1).
	\]
Since $\gamma\in \Omega$, we may choose a non-zero vector $\xi\in \fH\bksl \ker(T^\gamma)$. We claim that $\xi$ is separating for $T$.	To see this, let $\phi\in \M_d$ with $\phi(T)\neq 0$. An application of Theorem \ref{T:gleasontrick} with 
\[
n>(n_1-1)+(n_2-1)+\ldots+(n_d-1)
\]
shows that there is a polynomial $p$ such that $\phi-p\in\Ann_{\comm}(T)$, whence $\phi(T)=p(T)$. Furthermore, it is clear that $p$ can be chosen to be of the form
	\[
	p=\sum_{\alpha\in\Omega}c_\alpha x^\alpha .
	\]
	Because $p(T)\neq 0$, the set $S=\{\alpha\in\Omega:c_\alpha\neq 0\}$ is non-empty. Choose $\mu\in S$ with minimal length. It is easily seen that there is $\beta\in \bN^d $ such that $\mu+\beta=\gamma$ and $\alpha+\beta\notin \Omega$ for every $\alpha\in S\setminus \{\mu\}$. 
	Thus 
	\[
	T^\beta p(T)=c_\mu T^\gamma.
	\]
	Since $\xi\notin \ker T^\gamma$ and $c_\mu\neq 0$, it follows that $T^\beta p(T)\xi\neq 0$ and in particular $p(T)\xi\neq 0$. Therefore, $\phi(T)\xi\neq 0$, and we conclude that $\xi$ is separating for $T$.
	
	\qed
\end{example}

Unfortunately, separating vectors do not always exist, even for commuting nilpotent pairs on finite-dimensional spaces.

\begin{example}\label{E:maxcount}
	With respect to the usual orthonormal basis, we define linear operators on $\bC^3$ as
	\[ T_1 = \frac{1}{\sqrt{3}}\begin{bmatrix} 0 & 0 & 0 \\ 1 & 0 & -1 \\ 0 & 0 & 0 \end{bmatrix},
		\quad
	   T_2 = \frac{1}{\sqrt{3}}\begin{bmatrix} 0 & 0 & 0 \\ 0 & 0 & 1 \\ 0 & 0 & 0 \end{bmatrix}.
	\]
	Easy computations show that 
	\[
	T_1^2=T_1T_2=T_2T_1=T_2^2=0
	\]
	and that
	\[ T_1T_1^* + T_2T_2^* = \begin{bmatrix} 0 & 0 & 0 \\ 0 & 1 & 0 \\ 0 & 0 & 0\end{bmatrix}. \]
	Hence, $T=(T_1,T_2)$ is a nilpotent commuting row contraction.
	Let $\phi\in \M_2$. An application of Theorem \ref{T:gleasontrick} yields constants  $a,b,c\in \bC$ and multipliers $\phi_{12},\phi_{11},\phi_{22}\in \M_2$ such that
	\[
	\phi=a+bx_1+cx_2+\phi_{11}x_1^2+\phi_{12}x_1 x_2+\phi_{22}x_2^2.
	\]
	Then, we find
	\[ \phi(T_1,T_2)=a I+ bT_1+ cT_2 = 
	\begin{bmatrix} a& 0 & 0 \\ b/\sqrt{3} & a & (c-b)/\sqrt{3} \\ 0 & 0 & a \end{bmatrix}.
	 \]
	 Thus, $\phi(T_1,T_2)=0$ if and only if $a=b=c=0$.
	This shows that $\Ann_{\comm}(T)$ is the weak-$*$ closed ideal of $\M_2$ generated by $\{x_1^2,x_1x_2,x_2^2\}$. 
	
	We claim that no vector $v\in \bC^3$ is separating  for $T$. Indeed, write $v=(v_1,v_2,v_3)\in \bC^3$, and consider the polynomial $p$ defined as
	\[ 
	p= \begin{cases}
	 v_3x_1+(v_3-v_1)x_2 & \text{if }v_3\neq 0\text{ or }v_1\neq 0, \\
	x_1 & \text{otherwise}. 
	\end{cases}
	\]
	Then $p(T)v=0$ yet $p(T)\neq 0$, which means that $v$ is not separating. This establishes the claim that $T$ has no separating vector, and thus no cyclic vector either. Note however that the vector $(0,1,0)$ is easily verified to be cyclic for the pair $T^*=(T_1^*,T_2^*)$. Anticipating Theorem \ref{T:AmpliateMax} below, we also remark that the set $\{(1,0,0),(0,0,1)\}$ is separating for $T$. 
	\qed
\end{example}

Interestingly, separating vectors for single constrained contractions are known to exist in abundance \cite[Theorem II.3.6]{bercovici1988}. Inspection of the proofs of the results leading up to \cite[Theorem III.3.1]{bercovici1988} reveals that separating vectors play a crucial role there. Thus, the difficulties associated to the construction of cyclic decompositions in the multivariate context may be explained, in part, by the scarcity of separating vectors brought to light in Example \ref{E:maxcount}.

Another consequence of Example \ref{E:maxcount} is that nilpotent commuting $d$-tuples are not as nicely behaved as their single variable counterparts. For instance, by \cite[Theorem 1]{ADF1976} it is known that if $T$ is a nilpotent contraction, then there is a collection $\{T_\lambda:\lambda\in \Lambda\}$ of \emph{cyclic} nilpotent contractions such that
\[
\oplus_{\lambda\in \Lambda}T_\lambda \prec T.
\]
Trivially, a cyclic vector is necessarily separating, so the following lemma implies that such an operator $T$ must admit a separating vector. In light of the previous example, we conclude that a direct analogue of \cite[Theorem 1]{ADF1976} cannot hold in several variables.

\begin{lemma}\label{L:sep}
The following statements hold.
\begin{enumerate}

\item[\rm{(1)}] Let $T=(T_1,\ldots,T_d)$ and $T'=(T'_1,\ldots,T_d')$ be constrained commuting row contractions on some Hilbert spaces $\fH$ and $\fH'$. Assume that $T'\prec T$. If $T'$ has a separating vector, then so does $T$.

\item[\rm{(2)}] For each $n\in \bN$, let $T_n$  be a constrained commuting row contraction on some Hilbert space $\fH_n$. Assume that $T_n$ has a separating vector for every $n\in \bN$. Then, $\oplus_{n=1}^\infty T_n$ has a separating vector as well.
\end{enumerate}
\end{lemma}
\begin{proof}
For (1), let $X:\fH'\to \fH$ be an injective operator with dense range such that $XT_k'=T_k X$ for every $1\leq k\leq n$. Let $\xi\in \fH'$ be a separating vector for $T'$ and let $\theta\in \M_d$. Assume that $\theta(T)X\xi=0$. Then
\[
0=\theta(T)X\xi=X\theta(T')\xi
\]
and we infer that $\theta(T')\xi=0$ since $X$ is injective. Using that $\xi$ is separating for $T'$, we thus find that $\theta(T')=0$. Next, Lemma \ref{L:anntransf} and Equation (\ref{Eq:ann}) imply that 
\[
\Ann_{\comm}(T)=\Ann_{\comm}(T')
\]
so that $\theta(T)=0$. Hence $X\xi$ is separating for $T$.

Turning to (2), for each $n\in\bN$ we let $\xi_n\in \fH_n$ be a separating vector for $T_n$ with norm $1$. Define $\xi\in \oplus_{n=1}^\infty \fH_n$ as
\[
\xi=\left(\xi_1,\frac{1}{2}\xi_2,\ldots,\frac{1}{2^n}\xi_n,\ldots \right).
\]
It is readily verified that if $\theta\in \M_d$, then
\[
\theta(T)=\oplus_{n=1}^\infty \theta(T_n)
\]
and
\[
\theta(T)\xi=\left(\theta(T_1)\xi_1,\frac{1}{2}\theta(T_2)\xi_2,\ldots,\frac{1}{2^n}\theta(T_n)\xi_n,\ldots \right).
\]
The desired statement readily follows from these observations.

\end{proof}

Our final result below offers a counterpoint to Example \ref{E:maxcount} and shows that, at least for  nilpotent commuting row contractions, the existence of separating vectors can be insured provided that we allow for finite ampliations. Given a $d$-tuple of operators $T=(T_1,\ldots,T_d)$ on some Hilbert space $\fH$ and a positive integer $n\in \bN$, we define
\[
T^{(n)}=T\oplus T\oplus \ldots \oplus T
\]
which acts on the Hilbert space
\[
\fH^{(n)}=\fH\oplus \fH\oplus \ldots \oplus \fH.
\]
\begin{theorem}
	Let $N=(N_1,\dots,N_d)$ be a nilpotent commuting  row contraction on some Hilbert space $\fH$. Set
	\[
	\A=\bC[x_1,\dots,x_d]/(\bC[x_1,\dots,x_d]\cap \Ann_{\comm}(N)).
	\]
	Then, for any positive integer $s\geq \dim \A$ the ampliation $N^{(s)}$ has a dense set of separating vectors in $\fH^{(s)}$.
	\label{T:AmpliateMax}
\end{theorem}
\begin{proof}	

	Since $N$ is nilpotent, there are positive integers $m_1,\ldots,m_d$ such that $N_k^{m_k}=0$ for every $1\leq k\leq d$.
	An application of Theorem \ref{T:gleasontrick} with 
	\[
	n>(m_1-1)+(m_2-1)+\ldots+(m_d-1)
	\]
	shows that for any $\phi\in\M_d$, there is a polynomial $p$ such that $\phi(N)=p(N)$.
	For this reason, a vector $(\xi_1,\dots,\xi_s)\in \fH^{(s)}$ is separating for $N^{(s)}$ if whenever $p$ is a polynomial, the equalities 
	\[
	p(N)\xi_1=\dots=p(N)\xi_s=0
	\]
	are equivalent to $p(N)=0$. In other words, it suffices to consider polynomials rather than general multipliers.
		
	Note that $\A$ is a finite-dimensional vector space, and put $\delta=\dim \A$. For each $1\leq i\leq \delta$, choose a non-empty open subset $U_i\subset \fH$. Given a polynomial $p$, we denote by  $[p]$ its image in the quotient $\A$.	For every $h\in\fH$, set 
	\[
	\A_h=\{[p]\in\A:p(T)h=0\}.
	\]
We claim that there are vectors $\xi_1\in U_1,\ldots,\xi_\delta\in U_\delta$ such that $\cap_{i=1}^\delta \A_{\xi_i}=\{0\}$. Assume otherwise. For convenience, set $\A_{\xi_0}=\A$.

Fix a non-zero element $[p_1]\in\A_{\xi_0}$. Then, $p_1(T)\neq 0$  so we may choose $\xi_1\in U_1\setminus \ker p_1(T)$. Thus, $[p_1]\notin \A_{\xi_1}$ and in particular $\A_{\xi_1}$ is a proper subspace of $\A_{\xi_0}$. Suppose that for $1\leq i\leq \delta-1$ we have constructed $\xi_1\in U_1,\dots,\xi_i\in U_i$ with the property that $\cap_{j=0}^i \A_{\xi_j}$ is a proper subspace of $\cap_{j=0}^{i-1} \A_{\xi_j}$. By our standing assumption, we see that $\cap_{j=0}^i \A_{\xi_j}$ is non-zero, so there is a non-zero element $[p_{i+1}]\in\cap_{j=0}^{i} \A_{\xi_j}$. Correspondingly, there is a vector $\xi_{i+1}\in U_{i+1}\setminus \ker p_{i+1}(T)$. We see that $[p_{i+1}]\notin \A_{\xi_{i+1}}$, whence $\cap_{j=0}^{i+1}\A_{\xi_j}$ is a non-zero proper subspace of $\cap_{j=0}^{i} \A_{\xi_j}$. By induction, we obtain vectors $\xi_1\in U_1,\ldots,\xi_\delta\in U_\delta$ with the property that for every $1\leq i\leq \delta$, we have that $\cap_{j=0}^i \A_{\xi_j}$ is a non-zero proper subspace of $\cap_{j=0}^{i-1} \A_{\xi_j}$. This forces $\dim \A\geq \delta+1$, contrary to the choice of $\delta$. The claim is established.

Now, let $p$ be a polynomial such that $p(T)\xi_i=0$ for every $1\leq i\leq \delta$. Then, $[p]\in \cap_{i=0}^\delta \A_{\xi_i}$ and therefore $[p]=0$. This shows that $U_1\times \ldots \times U_\delta\subset \fH^{(\delta)}$ contains a separating vector for $T^{(\delta)}$. Since $U_1,\ldots,U_\delta$ were arbitrary non-empty open subsets of $\fH$, we conclude that the set $\Sigma\subset \fH^{(\delta)}$ of separating vectors for $T^{(\delta)}$ is dense in $\fH^{(\delta)}$. Finally, for $s>\delta$ it is readily verified that every vector in $\Sigma \oplus \fH^{(s-\delta)}$ is separating for $T^{(s)}$, so the result follows.
\end{proof}

We record the following simple reformulation.

\begin{corollary}\label{C:sepset}
Let $N=(N_1,\dots,N_d)$ be a  nilpotent commuting row contraction on some Hilbert space $\fH$. Then, there is a finite set of cardinality at most $\delta$ which is separating for $T$, where $\delta$ is the dimension of 
\[
	\bC[x_1,\dots,x_d]/(\bC[x_1,\dots,x_d]\cap \Ann_{\comm}(N)).
	\] 
\end{corollary}
\begin{proof}
This is an immediate consequence of Theorem \ref{T:AmpliateMax}.
\end{proof}

Comparing with Example \ref{Ex:Rectangle}, it is apparent that the upper bound on the cardinality of the separating set in the previous corollary is far from sharp.
In fact, the argument used in Example \ref{Ex:Rectangle} can be refined and extended to provide a better estimate.
Because we do not have a concrete use for it, we only state the result and leave the details of the proof to the interested reader. Recall that we denote by $[p]$ the image of a polynomial $p$ in the quotient $\A$.

\begin{quotation}
	\noindent 	Let $N=(N_1,\dots,N_d)$ be a  nilpotent commuting row contraction on some Hilbert space $\fH$.
	Let
  \[\Omega_e=\{\alpha\in\bN^d:T^\alpha\neq 0\text{ and }T^\alpha T_k=0 \text{ for every  }1\leq k\leq d\}. \]
	If $\{[x^\alpha]:\alpha\in\Omega_e\}$ is a linearly independent subset of $\A$, then $T$ admits a separating set of cardinality at most $\card \Omega_e$.
	\end{quotation}
The reader will note that in Example \ref{Ex:Rectangle}, we have $\card \Omega_e=1$, while for Example \ref{E:maxcount} we find $\card \Omega_e=2$.


\bibliography{/Users/raphaelclouatre/Dropbox/Research/Bibliography/biblio_main}

\def\cprime{$'$}
\begin{bibdiv}
\begin{biblist}

\bib{AM2000}{article}{
      author={Agler, Jim},
      author={McCarthy, John~E.},
       title={Complete {N}evanlinna-{P}ick kernels},
        date={2000},
        ISSN={0022-1236},
     journal={J. Funct. Anal.},
      volume={175},
      number={1},
       pages={111\ndash 124},
         url={http://dx.doi.org.uml.idm.oclc.org/10.1006/jfan.2000.3599},
      review={\MR{1774853}},
}

\bib{agler2002}{book}{
      author={Agler, Jim},
      author={McCarthy, John~E.},
       title={Pick interpolation and {H}ilbert function spaces},
      series={Graduate Studies in Mathematics},
   publisher={American Mathematical Society, Providence, RI},
        date={2002},
      volume={44},
        ISBN={0-8218-2898-3},
         url={http://dx.doi.org.proxy.lib.uwaterloo.ca/10.1090/gsm/044},
      review={\MR{1882259 (2003b:47001)}},
}

\bib{ADF1976}{article}{
      author={Apostol, C.},
      author={Douglas, R.~G.},
      author={Foia\c{s}, C.},
       title={Quasi-similar models for nilpotent operators},
        date={1976},
        ISSN={0002-9947},
     journal={Trans. Amer. Math. Soc.},
      volume={224},
      number={2},
       pages={407\ndash 415},
         url={https://doi-org.uml.idm.oclc.org/10.2307/1997485},
      review={\MR{0425651}},
}

\bib{AS1976}{article}{
      author={Apostol, Constantin},
      author={Stampfli, Joseph},
       title={On derivation ranges},
        date={1976},
        ISSN={0022-2518},
     journal={Indiana Univ. Math. J.},
      volume={25},
      number={9},
       pages={857\ndash 869},
         url={https://doi-org.uml.idm.oclc.org/10.1512/iumj.1976.25.25067},
      review={\MR{0412890}},
}

\bib{arveson1998}{article}{
      author={Arveson, William},
       title={Subalgebras of {$C^*$}-algebras. {III}. {M}ultivariable operator
  theory},
        date={1998},
        ISSN={0001-5962},
     journal={Acta Math.},
      volume={181},
      number={2},
       pages={159\ndash 228},
         url={http://dx.doi.org/10.1007/BF02392585},
      review={\MR{1668582 (2000e:47013)}},
}

\bib{arveson2000}{article}{
      author={Arveson, William},
       title={The curvature invariant of a {H}ilbert module over {${\bf
  C}[z_1,\cdots,z_d]$}},
        date={2000},
        ISSN={0075-4102},
     journal={J. Reine Angew. Math.},
      volume={522},
       pages={173\ndash 236},
         url={http://dx.doi.org/10.1515/crll.2000.037},
      review={\MR{1758582 (2003a:47013)}},
}

\bib{arveson2004}{article}{
      author={Arveson, William},
       title={The free cover of a row contraction},
        date={2004},
        ISSN={1431-0635},
     journal={Doc. Math.},
      volume={9},
       pages={137\ndash 161 (electronic)},
      review={\MR{2054985 (2005e:47013)}},
}

\bib{bercovici1988}{book}{
      author={Bercovici, Hari},
       title={Operator theory and arithmetic in {$H^\infty$}},
      series={Mathematical Surveys and Monographs},
   publisher={American Mathematical Society},
     address={Providence, RI},
        date={1988},
      volume={26},
        ISBN={0-8218-1528-8},
      review={\MR{954383 (90e:47001)}},
}

\bib{bercovici1991}{article}{
      author={Bercovici, Hari},
       title={The quasisimilarity orbits of invariant subspaces},
        date={1991},
        ISSN={0022-1236},
     journal={J. Funct. Anal.},
      volume={95},
      number={2},
       pages={344\ndash 363},
         url={http://dx.doi.org/10.1016/0022-1236(91)90033-2},
      review={\MR{1092130 (92e:47017)}},
}

\bib{BS1996}{article}{
      author={Bercovici, Hari},
      author={Smotzer, Thomas},
       title={Quasisimilarity of invariant subspaces for uniform {J}ordan
  operators of infinite multiplicity},
        date={1996},
        ISSN={0022-1236},
     journal={J. Funct. Anal.},
      volume={140},
      number={1},
       pages={87\ndash 99},
         url={https://doi-org.uml.idm.oclc.org/10.1006/jfan.1996.0099},
      review={\MR{1404575}},
}

\bib{BT1991}{article}{
      author={Bercovici, Hari},
      author={Tannenbaum, Allen},
       title={The invariant subspaces of a uniform {J}ordan operator},
        date={1991},
        ISSN={0022-247X},
     journal={J. Math. Anal. Appl.},
      volume={156},
      number={1},
       pages={220\ndash 230},
         url={https://doi-org.uml.idm.oclc.org/10.1016/0022-247X(91)90392-D},
      review={\MR{1102607}},
}

\bib{bhatta2005}{article}{
      author={Bhattacharyya, T.},
      author={Eschmeier, J.},
      author={Sarkar, J.},
       title={Characteristic function of a pure commuting contractive tuple},
        date={2005},
        ISSN={0378-620X},
     journal={Integral Equations Operator Theory},
      volume={53},
      number={1},
       pages={23\ndash 32},
  url={http://dx.doi.org.proxy.lib.uwaterloo.ca/10.1007/s00020-004-1309-5},
      review={\MR{2183594 (2006m:47012)}},
}

\bib{BRS1990}{article}{
      author={Blecher, David~P.},
      author={Ruan, Zhong-Jin},
      author={Sinclair, Allan~M.},
       title={A characterization of operator algebras},
        date={1990},
        ISSN={0022-1236},
     journal={J. Funct. Anal.},
      volume={89},
      number={1},
       pages={188\ndash 201},
         url={http://dx.doi.org.uml.idm.oclc.org/10.1016/0022-1236(90)90010-I},
      review={\MR{1040962}},
}

\bib{chenguo2003}{book}{
      author={Chen, Xiaoman},
      author={Guo, Kunyu},
       title={Analytic {H}ilbert modules},
      series={Chapman \& Hall/CRC Research Notes in Mathematics},
   publisher={Chapman \& Hall/CRC, Boca Raton, FL},
        date={2003},
      volume={433},
        ISBN={1-58488-399-5},
         url={https://doi-org.uml.idm.oclc.org/10.1201/9780203008836},
      review={\MR{1988884}},
}

\bib{clouatre2013invsub}{article}{
      author={Clou\^atre, Rapha\"el},
       title={Quasisimilarity of invariant subspaces for {$C_0$} operators with
  multiplicity two},
        date={2013},
        ISSN={0379-4024},
     journal={J. Operator Theory},
      volume={70},
      number={2},
       pages={495\ndash 511},
         url={https://doi-org.uml.idm.oclc.org/10.7900/jot.2011sep19.1937},
      review={\MR{3138367}},
}

\bib{clouatre2014invsub}{article}{
      author={Clou\^atre, Rapha\"el},
       title={Quasiaffine orbits of invariant subspaces for uniform {J}ordan
  operators},
        date={2014},
        ISSN={0022-1236},
     journal={J. Funct. Anal.},
      volume={266},
      number={7},
       pages={4101\ndash 4114},
         url={https://doi-org.uml.idm.oclc.org/10.1016/j.jfa.2014.01.019},
      review={\MR{3170203}},
}

\bib{CD2016abscont}{article}{
      author={Clou\^atre, Rapha\"el},
      author={Davidson, Kenneth~R.},
       title={Absolute continuity for commuting row contractions},
        date={2016},
        ISSN={0022-1236},
     journal={J. Funct. Anal.},
      volume={271},
      number={3},
       pages={620\ndash 641},
         url={http://dx.doi.org.uml.idm.oclc.org/10.1016/j.jfa.2016.04.012},
      review={\MR{3506960}},
}

\bib{CD2016duality}{article}{
      author={Clou\^atre, Rapha\"el},
      author={Davidson, Kenneth~R.},
       title={Duality, convexity and peak interpolation in the
  {D}rury-{A}rveson space},
        date={2016},
        ISSN={0001-8708},
     journal={Adv. Math.},
      volume={295},
       pages={90\ndash 149},
         url={http://dx.doi.org.uml.idm.oclc.org/10.1016/j.aim.2016.02.035},
      review={\MR{3488033}},
}

\bib{CH2018}{article}{
      author={Clou\^atre, Rapha\"el},
      author={Hartz, Michael},
       title={Multiplier algebras of complete {N}evanlinna-{P}ick spaces:
  dilations, boundary representations and hyperrigidity},
        date={2018},
        ISSN={0022-1236},
     journal={J. Funct. Anal.},
      volume={274},
      number={6},
       pages={1690\ndash 1738},
         url={https://doi-org.uml.idm.oclc.org/10.1016/j.jfa.2017.10.008},
      review={\MR{3758546}},
}

\bib{davidson1998alg}{article}{
      author={Davidson, Kenneth~R.},
      author={Pitts, David~R.},
       title={The algebraic structure of non-commutative analytic {T}oeplitz
  algebras},
        date={1998},
        ISSN={0025-5831},
     journal={Math. Ann.},
      volume={311},
      number={2},
       pages={275\ndash 303},
         url={http://dx.doi.org.proxy.lib.uwaterloo.ca/10.1007/s002080050188},
      review={\MR{1625750 (2001c:47082)}},
}

\bib{davidson1998}{article}{
      author={Davidson, Kenneth~R.},
      author={Pitts, David~R.},
       title={Nevanlinna-{P}ick interpolation for non-commutative analytic
  {T}oeplitz algebras},
        date={1998},
        ISSN={0378-620X},
     journal={Integral Equations Operator Theory},
      volume={31},
      number={3},
       pages={321\ndash 337},
         url={http://dx.doi.org.proxy.lib.uwaterloo.ca/10.1007/BF01195123},
      review={\MR{1627901 (2000g:47016)}},
}

\bib{davidson1999}{article}{
      author={Davidson, Kenneth~R.},
      author={Pitts, David~R.},
       title={Invariant subspaces and hyper-reflexivity for free semigroup
  algebras},
        date={1999},
        ISSN={0024-6115},
     journal={Proc. London Math. Soc. (3)},
      volume={78},
      number={2},
       pages={401\ndash 430},
         url={http://dx.doi.org/10.1112/S002461159900180X},
      review={\MR{1665248 (2000k:47005)}},
}

\bib{DRS2015}{article}{
      author={Davidson, Kenneth~R.},
      author={Ramsey, Christopher},
      author={Shalit, Orr~Moshe},
       title={Operator algebras for analytic varieties},
        date={2015},
        ISSN={0002-9947},
     journal={Trans. Amer. Math. Soc.},
      volume={367},
      number={2},
       pages={1121\ndash 1150},
  url={http://dx.doi.org.uml.idm.oclc.org/10.1090/S0002-9947-2014-05888-1},
      review={\MR{3280039}},
}

\bib{gleason2005}{article}{
      author={Gleason, Jim},
      author={Richter, Stefan},
      author={Sundberg, Carl},
       title={On the index of invariant subspaces in spaces of analytic
  functions of several complex variables},
        date={2005},
        ISSN={0075-4102},
     journal={J. Reine Angew. Math.},
      volume={587},
       pages={49\ndash 76},
         url={https://doi-org.uml.idm.oclc.org/10.1515/crll.2005.2005.587.49},
      review={\MR{2186975}},
}

\bib{griffith1969}{article}{
      author={Griffith, Phillip},
       title={On the decomposition of modules and generalized left uniserial
  rings},
        date={1969/1970},
        ISSN={0025-5831},
     journal={Math. Ann.},
      volume={184},
       pages={300\ndash 308},
         url={https://doi-org.uml.idm.oclc.org/10.1007/BF01350858},
      review={\MR{0257136}},
}

\bib{GHX04}{article}{
      author={Guo, Kunyu},
      author={Hu, Junyun},
      author={Xu, Xianmin},
       title={{Toeplitz} algebras, subnormal tuples and rigidity on reproducing
  {$\bold C[z_1,\dots,z_d]$}-modules},
        date={2004},
        ISSN={0022-1236},
     journal={J. Funct. Anal.},
      volume={210},
      number={1},
       pages={214\ndash 247},
      review={\MR{2052120 (2005a:47007)}},
}

\bib{kribs2001Ld}{article}{
      author={Kribs, David~W.},
       title={Factoring in non-commutative analytic {T}oeplitz algebras},
        date={2001},
        ISSN={0379-4024},
     journal={J. Operator Theory},
      volume={45},
      number={1},
       pages={175\ndash 193},
      review={\MR{1823067}},
}

\bib{li1999}{article}{
      author={Li, Wing~Suet},
      author={M{\"u}ller, Vladim{\'{\i}}r},
       title={Invariant subspaces of nilpotent operators and {LR}-sequences},
        date={1999},
        ISSN={0378-620X},
     journal={Integral Equations Operator Theory},
      volume={34},
      number={2},
       pages={197\ndash 226},
         url={http://dx.doi.org/10.1007/BF01236472},
      review={\MR{1694708 (2001d:47013)}},
}

\bib{muller1993}{article}{
      author={M{\"u}ller, V.},
      author={Vasilescu, F.-H.},
       title={Standard models for some commuting multioperators},
        date={1993},
        ISSN={0002-9939},
     journal={Proc. Amer. Math. Soc.},
      volume={117},
      number={4},
       pages={979\ndash 989},
         url={http://dx.doi.org.proxy.lib.uwaterloo.ca/10.2307/2159525},
      review={\MR{1112498 (93e:47016)}},
}

\bib{parrott1970}{article}{
      author={Parrott, Stephen},
       title={Unitary dilations for commuting contractions},
        date={1970},
        ISSN={0030-8730},
     journal={Pacific J. Math.},
      volume={34},
       pages={481\ndash 490},
      review={\MR{0268710 (42 \#3607)}},
}

\bib{popescu1989char}{article}{
      author={Popescu, Gelu},
       title={Characteristic functions for infinite sequences of noncommuting
  operators},
        date={1989},
        ISSN={0379-4024},
     journal={J. Operator Theory},
      volume={22},
      number={1},
       pages={51\ndash 71},
      review={\MR{1026074 (91m:47012)}},
}

\bib{popescu1989}{article}{
      author={Popescu, Gelu},
       title={Isometric dilations for infinite sequences of noncommuting
  operators},
        date={1989},
        ISSN={0002-9947},
     journal={Trans. Amer. Math. Soc.},
      volume={316},
      number={2},
       pages={523\ndash 536},
         url={http://dx.doi.org.proxy.lib.uwaterloo.ca/10.2307/2001359},
      review={\MR{972704 (90c:47006)}},
}

\bib{popescu1991}{article}{
      author={Popescu, Gelu},
       title={von {N}eumann inequality for {$(B({\scr H})^n)_1$}},
        date={1991},
        ISSN={0025-5521},
     journal={Math. Scand.},
      volume={68},
      number={2},
       pages={292\ndash 304},
      review={\MR{1129595 (92k:47073)}},
}

\bib{popescu1996}{article}{
      author={Popescu, Gelu},
       title={Non-commutative disc algebras and their representations},
        date={1996},
        ISSN={0002-9939},
     journal={Proc. Amer. Math. Soc.},
      volume={124},
      number={7},
       pages={2137\ndash 2148},
  url={http://dx.doi.org.proxy.lib.uwaterloo.ca/10.1090/S0002-9939-96-03514-9},
      review={\MR{1343719 (96k:47077)}},
}

\bib{popescu2006}{article}{
      author={Popescu, Gelu},
       title={Operator theory on noncommutative varieties},
        date={2006},
        ISSN={0022-2518},
     journal={Indiana Univ. Math. J.},
      volume={55},
      number={2},
       pages={389\ndash 442},
  url={http://dx.doi.org.proxy.lib.uwaterloo.ca/10.1512/iumj.2006.55.2771},
      review={\MR{2225440 (2007m:47008)}},
}

\bib{popescu2014sim}{article}{
      author={Popescu, Gelu},
       title={Similarity problems in noncommutative polydomains},
        date={2014},
        ISSN={0022-1236},
     journal={J. Funct. Anal.},
      volume={267},
      number={11},
       pages={4446\ndash 4498},
         url={https://doi-org.uml.idm.oclc.org/10.1016/j.jfa.2014.09.023},
      review={\MR{3269883}},
}

\bib{nagy2010}{book}{
      author={Sz.-Nagy, B{\'e}la},
      author={Foias, Ciprian},
      author={Bercovici, Hari},
      author={K{\'e}rchy, L{\'a}szl{\'o}},
       title={Harmonic analysis of operators on {H}ilbert space},
     edition={enlarged},
      series={Universitext},
   publisher={Springer},
     address={New York},
        date={2010},
        ISBN={978-1-4419-6093-1},
         url={http://dx.doi.org/10.1007/978-1-4419-6094-8},
      review={\MR{2760647 (2012b:47001)}},
}

\bib{varopoulos1974}{article}{
      author={Varopoulos, N.~Th.},
       title={On an inequality of von {N}eumann and an application of the
  metric theory of tensor products to operators theory},
        date={1974},
     journal={J. Functional Analysis},
      volume={16},
       pages={83\ndash 100},
      review={\MR{0355642 (50 \#8116)}},
}

\end{biblist}
\end{bibdiv}
\bibliographystyle{plain}


\end{document}